\documentclass[10pt,english]{amsart}
\usepackage{dsfont}
\usepackage{url}
\usepackage[utf8]{inputenc}
\usepackage[T1]{fontenc}
\usepackage{lmodern}
\usepackage{babel}
\usepackage{mathtools}  
\usepackage{amssymb}
\usepackage{lipsum}
\usepackage{mathrsfs}
\usepackage{color}
\usepackage{mathabx}

\newtheorem{theorem}{Theorem}[section]

\newtheorem{definition}[theorem]{Definition}

\newtheorem{lemma}[theorem]{Lemma}

\newtheorem{proposition}[theorem]{Proposition}

\newcommand{\R}{\mathbb{R}}

\newcommand{\bel}{\begin{equation} \label}
\newcommand{\ee}{\end{equation}}
\newcommand{\M}{\mathrm{M}}
\newcommand{\g}{\mathrm{g}}

\newcommand{\dd}{\mathrm{d}}
\newcommand{\dv}{\,\mathrm{dv}^{n}}

\newcommand{\dvv}{\,\mathrm{dv}^{2n-1}}
\newcommand{\ds}{\,\mathrm{d\sigma}^{n-1}}

\newcommand{\dss}{\,\mathrm{d\sigma}^{2n-2}}
\newcommand{\I}{\mathcal{I}}
\newcommand{\h}{\mathcal{H}}

\newcommand{\p}{\partial}
\newcommand{\bb}{\beta^*}
\newcommand{\norm}[1]{\|#1\|}
\newcommand{\Abs}[1]{\left|#1\right|}
\newcommand{\abs}[1]{|#1|}
\newcommand{\set}[1]{\left\{#1\right\}}
\newcommand{\para}[1]{\left(#1\right)}
\newcommand{\cro}[1]{\left[#1\right]}

\newcommand{\seq}[1]{\langle #1\rangle}

\newcommand{\To}{\longrightarrow}
\newcommand{\bphi}{\boldsymbol{\phi}}
\newcommand{\bpsi}{\boldsymbol{\psi}}
\newcommand{\blambda}{\boldsymbol{\lambda}}
\newcommand{\w}{\varphi_{1,\tau}^*}
\newcommand{\ww}{\varphi_{2,\tau}^*}
\usepackage{graphicx}
\newcommand{\dive}{\textrm{div}}

\begin{document}

\title[A Borg-Levinson theorem for magnetic Schr\"odinger operators]{A Borg-Levinson theorem for magnetic Schr\"odinger operators on a Riemannian manifold}

\author[Bellassoued]{Mourad Bellassoued}
\address{Universit\'e de Tunis El Manar, Ecole Nationale d'Ing\'enieurs de Tunis, LAMSIN, BP 37, 1002 Tunis Le Belv\'ed\`re, Tunisia.}
\email{mourad.bellassoued@enit.utm.tn}

\author[Choulli]{Mourad Choulli}
\address{Institut \'Elie Cartan de Lorraine, UMR CNRS 7502, Universit\'e de Lorraine, F-57045 Metz cedex 1, France.}
\thanks{MC and YK are supported by the grant ANR-17-CE40-0029 of the French National Research Agency ANR (project MultiOnde).}
\email{mourad.choulli@univ-lorraine.fr}

\author[Dos Santos Ferreira]{David Dos Santos Ferreira}
\address{Institut \'Elie Cartan de Lorraine, UMR CNRS 7502, \'equipe SPHINX, INRIA, Universit\'e de Lorraine, F-54506 Vandoeuvre-lès-Nancy Cedex, France.}
\email{ddsf@math.cnrs.fr}
\thanks{DDSF is partly supported by ANR grant \textit{iproblems}.}

\author[Kian]{Yavar Kian}
\address{Aix Marseille Univ, Universit\'e de Toulon, CNRS, CPT, Marseille, France.}
\email{yavar.kian@univ-amu.fr}

\author[Stefanov]{Plamen Stefanov}
\address{Department of Mathematics, Purdue University, West Lafayette, IN 47907, USA.}
\email{Plamen-Stefanov@purdue.edu}
\thanks{PS  partly supported by  NSF  Grant DMS-1600327}

\date{}
\subjclass[2010]{Primary 35R30, 35J10, Secondary: 35P99. } 
\keywords{Borg-Levinson type theorem, magnetic Schr\"odinger operator, simple Riemannian manifold, uniqueness, stability estimate.}
\begin{abstract}
This article is concerned with uniqueness and stability issues for the inverse spectral problem of recovering the magnetic field and the electric potential in a Riemannian manifold from some asymptotic knowledge of the boundary spectral data of the corresponding Schr\"odinger operator under Dirichlet boundary conditions. The spectral data consist of some asymptotic knowledge of a subset of eigenvalues and Neumann traces of the associated eigenfunctions of the magnetic Laplacian. We also address the same question for Schr\"odinger operators under Neumann boundary conditions, in which case we measure the Dirichlet traces of eigenfunctions. In our results we characterize the uniqueness of the magnetic field from a rate of growth of the eigenvalues, combined with suitable asymptotic properties of boundary observation of eigenfunctions, of the associated magnetic Schr\"odinger operator. To our best knowledge this is the first result proving uniqueness from such general asymptotic behavior of boundary spectral data.
\end{abstract}

\maketitle

\section{Introduction and main results}


\subsection{Statement of the problem}

Let $\M=(\M,\,\g)$ be a smooth and compact
Riemannian manifold of dimension $n\geq2$ and with boundary $\partial \M$. We denote the Laplace-Beltrami operator associated to the Riemannian metric $\g$ by $\Delta$. In local coordinates,
the metric reads $\g=(\g_{jk})$, and the Laplace-Beltrami operator $\Delta$ is given by
$$
\Delta=\frac{1}{\sqrt{\abs{\g}}}\sum_{j,k=1}^n\frac{\p}{\p
x_j}\para{\sqrt{\abs{\g}}\,\g^{jk}\frac{\p}{\p x_k}}.
$$
Here $(\g^{jk})$ is the inverse of the metric $\g$ and $\abs{\g}=\det(\g_{jk})$.  

Given a couple of magnetic and electric potentials $B=(A,q)$, where $q\in L^\infty(\M)$ is real-valued, and $A=a_jdx^j$ is a covector field (1-form) with real-valued coefficients, $a_j\in W^{1,\infty}(\M)$, we consider the magnetic Laplacian
\begin{align}\label{1.1}
\h_{B}&=\frac{1}{\sqrt{\abs{\g}}}\sum_{j,k=1}^n\para{\frac{1}{i}\frac{\p}{\p
x_j}+a_j}\sqrt{\abs{\g}}\,\g^{jk}\para{\frac{1}{i}\frac{\p}{\p x_k}+a_k}+q\cr
&=-\Delta-2i\,A\cdot\nabla-i\,\delta A+|A|^2+q\cr
&:=-\Delta_A+q.
\end{align}
Here,  the dot product is in the metric with $A$ and $\nabla$ considered as covectors, $\delta$ is the coderivative (codifferential) operator, corresponding to the divergence with identifying vectors and covectors, which 
sends $1$-forms  to a functions by the formula 
$$
\delta A=\frac{1}{\sqrt{\abs{\g}}}\sum_{j,k=1}^n \frac{\p}{\p x^j}\para{\g^{jk}\sqrt{\abs{\g}} a_k},
$$
and we recall that, for $A=a_jdx^j$, we have $|A|^2=\g^{jk}a_ja_k$. 

For $B=(A,q)$ with $q\in L^\infty(\M)$ and $A=a_jdx^j$, $a_j\in W^{1,\infty}(\M)$,  define on $L^2(\M)$ the unbounded  self-adjoint operator $H_B$ as follows
\begin{equation}\label{1.2}
H_Bu=\mathcal{H}_B u
\end{equation}
and
\begin{equation}\label{1.3}
\mathscr{D}(H_B)=\set{u\in H_0^1(\M),\,\, -\Delta _Au+qu\in L^2(\M)}.
\end{equation}
Here $H^k(\M)$, denotes the standard  definition of the Sobolev spaces.

The operator $H_B$ is self adjoint and has compact resolvent,  therefore its spectrum $\sigma (H_B)$ consists in a sequence $\blambda _B=(\lambda _{B,k})$ of real eigenvalues, counted according to their multiplicities, so that
\begin{equation}\label{1.4}
-\infty < \lambda_{B,1}\le \lambda_{B,2} \le \ldots \le \lambda_{B,k}\rightarrow +\infty \quad \textrm{as}\; k\rightarrow \infty .
\end{equation}
In the sequel $\bphi_B= (\phi_{B,k})$ denotes an orthonormal basis of $L^2(\M)$ consisting in eigenfunctions with $\phi_{B,k}$ associated to $\lambda_{B,k}$, for each $k$.\\
In the rest of this text, we often use the following notation, where $k\ge 1$,
$$
\psi_{B,k}= \para{\partial_\nu+iA(\nu)}\phi_{B,k},\quad \textrm{on}\;  \p \M
$$
and $\bpsi _B=(\psi_{B,k})$, where $\nu$ the outward unit normal vector field on $\p M$ with respect to the metric $\g$.

We address the question of whether one can recover, in some suitable sense, the magnetic field $A$ and the potential $q$ from some asymptotic knowledge of the boundary spectral data 
$(\blambda_B, \bpsi_B)$ with $B=(A,q)$. As for most inverse problems, the main issues are uniqueness and stability.

\subsection{Obstruction to uniqueness}
We recall that there is an obstruction to the recovery of the electromagnetic potential $B$ from the boundary spectral data $(\blambda_B, \bpsi_B)$. Indeed, let $B=(A,q)$, and let $V\in \mathcal{C}^1(\M)$ be such that $V_{|\p\M}=0$ and set $\widecheck{B}=(A+dV,q)$. Then it is straightforward to check that
\begin{equation}\label{1.5}
e^{-iV}\mathcal{H}_Be^{iV}=\mathcal H_{\widecheck{B}},\quad (\blambda _B,\bpsi _B)=(\blambda_{\widecheck{B}}, \bpsi_{\widecheck{B}}).
\end{equation}
Therefore, the magnetic potential $A$  cannot be uniquely determined by
the boundary spectral data $(\blambda _B,\bpsi _B)$ and our inverse problem needs to be stated differently. 

According to \cite{Sh}, for every covector $A \in H^k(\M,T^*\M)$, there exist uniquely determined $A^s \in H^k(\M,T^*\M)$ and $V \in H^{k+1}(\M)$ such that
\begin{equation}\label{1.6}
A=A^s+dV,\quad \delta A^s=0,\quad V|_{\p\M}=0.
\end{equation}
 Following the well established terminology,  $A^s$ and $dV$ are called respectively the solenoidal and potential parts of the covector $A$. In view of the obstruction described above, the best one can expect is the simultaneous recovery of $A^s$ and $q$ from some knowledge of the boundary spectral data $(\blambda_B, \bpsi_B)$. From now on, we focus our attention on this problem.

\subsection{Known results}

There is a vast literature devoted to inverse spectral problems in one dimension. We refer for instance to  the pioneer works by Ambartsumian \cite{A}, Borg \cite{Bo}, Levinson \cite{L}, Gel'fand and Levitan  \cite{GL}.  The first multidimensional uniqueness result of this type is due to Nachman, Sylvester and Uhlmann \cite{NSU}  for the operator $-\Delta+q$ with $g$ Euclidean. They showed that $q$ is uniquely determined by the Dirichlet eigenvalues and the traces of the normal derivatives of the corresponding eigenfunctions.
 Later,  Isozaki \cite{I}
proved that if finitely many eigenvalues and eigenfunctions are omitted, we still have uniqueness. 
 In \cite{Su}, Sun studied, in this context, the recovery of magnetic Schr\"odinger operator from boundary measurements. The result of \cite{Su} requires an assumption of smallness of the magnetic field. This assumption was removed by Nakamura, Sun and Uhlmann in \cite{NSuU} as a consequence of their result on the Calder\'on's problem for such operators. Developing further Isozaki's approach, Choulli and  Stefanov \cite{CS}  gave a generalization of Isozaki's uniqueness result together with a H\"older stability estimate with respect to appropriate metrics for the spectral data. We mention that, following a remark of Isozaki which goes back to \cite{I}, the uniqueness and stability results of \cite{CS} were stated with only some  asymptotic closeness of the boundary spectral data. We mention also the work of \cite{ CK1, CK2}, dealing with recovery of general non-smooth coefficients from the full boundary spectral data, and the work \cite{KKS} who have considered a similar inverse spectral problem  for Schr\"odinger operators in an infinite cylindrical waveguide.
 
Another approach for getting uniqueness in the spectral inverse problem for the Laplace-Beltrami operator was introduced by Belishev \cite{Bel87} and Belishev and Kurylev \cite{Bel92}. This approach consists in  reducing the inverse spectral problem under consideration into an inverse hyperbolic problem for which one can apply the so called boundary control method. This method allows to consider the trace of the normal derivative of eigenfunctions only in a part of the boundary. We refer to \cite{Bel92, KKL, KK,Lassas2010,LO} and \cite{KOM} in the case of non-smooth coefficients. We mention that none of these papers considered this problem with observations corresponding to some asymptotic knowledge of the boundary spectral data. Actually, to our best knowledge, beside the present paper, there is no other results dealing with inverse spectral problem on non flat manifolds with data similar to the one considered by \cite{CS,KKS}.

One of the first stability estimate  for inverse spectral problems was established by Alessandrini and Sylvester \cite{AS}.   This result was reformulated by the second author in a more precise way in \cite{Ch}. A similar result  in the case of the Laplace-Beltrami operator was proved by the first and the third authors in \cite{BD} using the idea introduced in \cite{AS}. With the help of a result quantifying the uniqueness of continuation for a Cauchy problem with data on a part of the boundary for a wave equation, the first two authors and Yamamoto \cite{BCY} proved a double logarithmic stability estimate under the assumption that the potential is known near the boundary.  In \cite{CS}, the second and the last authors  provided one of the first H\"older type stability estimate for the  multi-dimensional Borg-Levinson theorem of determining
the potential from some asymptotic knowledge of the boundary spectral data of the associated Schr\"odinger operator. In \cite{KKS}, the fourth author, Kavian and Soccorsi  proved a similar result for an inverse spectral problem in an infinite cylindrical waveguide.

\subsection{Preliminaries}

We briefly recall some notations and known results in Riemannian geometry. We refer for instance to \cite{[Jost]} for more details. By Riemannian manifold with boundary, we mean a $C^\infty$-smooth manifold with boundary in the usual sense, endowed with a metric $\g$.

As before $\M$ denotes a compact Riemannian manifold of dimension $n\geq 2$. Fix a local coordinate system $x=\para{x^1,\ldots,x^n}$ and let $\para{\p_1,\dots,\p_n}$ be the corresponding tangent vector fields. For $x\in \M$, the inner product and the norm on the tangent space $T_x\M$ are given by
\begin{gather*}
\g(X,Y)=\seq{X,Y}=\sum_{j,k=1}^n\g_{jk}X^jY^k, \\
\abs{X}=\seq{X,X}^{1/2},\qquad  X=\sum_{i=1}^nX^i\p_i,\quad Y=\sum_{i=1}^n
Y^i\p_i.
\end{gather*}
The cotangent space $T_x^*\M$ is the dual of $T_x\M$. Its elements are
called covectors or one-forms. The disjoint union of the tangent spaces 
\[T\M=\underset{x\in\M}{\bigcup} T_x\M\]
 is called the tangent bundle of $\M$. Similarly, the cotangent bundle $T^*\M$ is the disjoint union of the spaces $T^*_x\M$, $x\in\M$. A $1$-form $A$ on the manifold $\M$ is a function that assigns to each point $x\in\M$ a covector $A(x)\in T^*_x\M$. 

An example of a $1$-form is the differential of a function $f \in\mathcal{C}^\infty(\M)$, which is defined by
$$
df_x(X)=\sum_{j=1}^n X^j\frac{\p f}{\p x_j},\quad X=\sum_{j=1}^nX^j\p_j.
$$
Hence $f$ defines the mapping $df : T\M\to\R$, which is called the differential of $f$ given by
$$
df(x,X)=df_x(X).
$$
In local coordinates,
$$
df=\sum_{j=1}^n\p_jfdx^j.
$$
where $(dx^1,\dots,dx^n)$ is the basis in the space $T^*_x\M$, dual to the basis $(\p_1,\dots,\p_n)$.

The Riemannian metric $\g$ induces a natural isomorphism $\imath : T_x\M\to T^*_x\M$ given
by $\iota(X) = \seq{X,\cdot}$. For $X\in T_x\M$ denote $X^\flat= \imath(X)$, and similarly for $A\in T^*_x\M$ we denote $A^\sharp=\imath^{-1}(A)$, $\imath$ and $\imath^{-1}$ are called musical isomorphisms. The \textit{sharp} operator is given by
\begin{equation}\label{1.7}
T^*_x\M\To T_x\M,\quad A\longmapsto A^\sharp,
\end{equation}
given in local coordinates by
\begin{equation}\label{1.8}
(a_jdx^j)^\sharp=a^j\p_j,\quad a^j=\sum_{k=1}^n\g^{jk}a_k.
\end{equation}

Define the inner product of $1$-forms in $T^*_x\M$ by 
\begin{equation}\label{1.9}
\seq{A,B}=\seq{A^\sharp,B^\sharp}=\sum_{j,k=1}^n \g^{jk}a_jb_k=\sum_{j,k=1}^n \g_{jk}a^jb^k.
\end{equation}
The metric tensor $\g$ induces the Riemannian volume 
\[ 
\dv=|\g|^{1/2}\dd x_1\wedge\cdots \wedge \dd x_n.
\]
 We denote by $L^2(\M)$ the completion
of $\mathcal{C}^\infty(\M)$ endowed with the usual inner product
$$
\para{f_1,f_2}=\int_\M f_1(x) \overline{f_2(x)} \dv,\qquad  f_1,f_2\in\mathcal{C}^\infty(\M).
$$
A section of a vector bundle $E$ over the Riemannian manifold $\M$ is a $C^\infty$ map $\mathfrak{s}:\M\to E$ such that for each $x\in\M$, $\mathfrak{s}(x)$ belongs to the fiber over $x$. We denote by $\mathcal{C}^\infty(\M,E)$ the space of smooth sections of the vector bundle $E$. According to this definition,  $\mathcal{C}^\infty(\M,T\M)$ denotes the space of vector fields on $\M$ and $\mathcal{C}^\infty(\M,T^*\M)$ denotes the space of $1$-forms on $\M$. Similarly, we may define the spaces $L^2(\M,T^*\M)$ (resp. $L^2(\M,T\M)$) of square integrable $1$-forms (resp. vectors) by using the inner product
\begin{equation}\label{1.10}
\para{A,B}=\int_\M\seq{A,\overline{B}}\dv,\quad A,B \in T^*\M.
\end{equation}

Define the Sobolev space $H^k(\M)$ as the completion of $\mathcal{C}^\infty(\M)$ with respect to the norm 
$$
\norm{f}^2_{H^k(\M)}=\norm{f}^2_{L^2(\M)}+\sum_{k=1}^n\|\nabla^k f\|^2_{L^2(\M,T^k\M)},
$$
where $\nabla^k$ is the covariant differential of $f$ in the metric $\g$.
If $f$ is a $\mathcal{C}^\infty$ function on $\M$, then $\nabla f$ is the vector field defined by
$$
X(f)=\seq{\nabla f,X},
$$
for every vector field $X$ on $\M$. In the local coordinates system, the last identity can be rewritten in the form
\begin{equation}\label{1.11}
\nabla f=\sum_{i,j=1}^n\g^{ij}\frac{\p f}{\p x_i}\p_j=(df)^\sharp.
\end{equation}
The normal derivative of a function $u$ is given by the formula
\begin{equation}\label{1.12}
\p_\nu u:=\seq{\nabla u,\nu}=\sum_{j,k=1}^n\g^{jk}\nu_j\frac{\p u}{\p x_k},
\end{equation}
where $\nu$ is the unit outward vector field to $\p \M$.

Likewise, we say that a $1$-form $A=a_jdx^j$ belongs to $H^k(\M,T^*\M)$ if each  $a_j\in H^k(\M)$. The space $H^k(\M,T^*\M)$ is a Hilbert space when it is endowed with the norm
$$
\norm{A}_{H^k(\M,T^*\M)}=\left(\sum_{j=1}^n\norm{a_j}_{H^k(\M)}^2\right)^{\frac{1}{2}}.
$$
As usual, the vector space of smooth $2$-forms on $\M$ is denoted by $\Omega^2(\M)$. In local coordinates, a $2$-form $\omega$ is represented as
$$
\omega=\sum_{j,k=1}^n\omega_{jk}dx^j\wedge dx^k,
$$
where $\omega_{jk}$ are  real-valued functions on $\M$. Similarly as before, $\omega$ is in $H^s(\M,\Omega^2(\M))$, $s\in\R$, if $\omega_{jk} \in H^s(\M)$ for each $j,k$. Additionally, $H^s(\M,\Omega^2(\M))$ is a Hilbert space for the norm
$$
\norm{\omega}_{H^s(\M,\Omega^2(\M))}=\left(\sum_{j,k}\norm{\omega_{jk}}_{H^s(\M)}^2\right)^{\frac{1}{2}}.
$$

In the rest of this text,  the scalar product of $L^2(\p \M)$ is also denoted by $\langle \cdot\, ,\cdot \, \rangle$:
\begin{equation}\label{1.13}
\langle f_1,f_2\rangle = \int_{\p \M} f_1(x)\, \overline{f_2(x)}\,\ds
\end{equation}
where $\ds$ is the volume form of $\p\M$.
\subsection{Main results}
Prior to the statement of our main results, we introduce the notion of simple manifolds \cite{SU}. We say that the boundary $\p\M$ is strictly convex if  the second fundamental form  is positive-definite for any $x \in \p\M$.
\begin{definition}
A manifold $\M$ is simple if $\p \M$ is
strictly convex and, for any $x\in \M$, the exponential map
$\exp_x:\exp_x^{-1}(\M)\To \M$ is a diffeomorphism. 
\end{definition}

Note that if $\M$ is simple, then it is diffeomorphic to a ball, and every two points can be connected by a unique minimizing geodesic depending smoothly on its endpoints. Also, one can extend it to a simple manifold $\M_{1}$ such that $\M_1^{\textrm{int}}\supset\M$.

We now introduce the admissible sets of magnetic potentials $A$ and electric potentials $q$. Set
$$
\mathscr{B}=W^{2,\infty}(\M,T^* \M)\oplus L^\infty (\M).
$$
We endow $\mathscr{B}$ with its natural norm
$$
\|B\|_{\mathscr{B}}=\|A\|_{W^{2,\infty}(\M,T^*\M)}+\|q\|_{L^\infty (\M)}.
$$
For $r>0$, set
\begin{equation}\label{1.14}
\mathscr{B}_r= \set{B=(A,q)\in\mathscr{B}, \,\,\norm{B}_{\mathscr{B}}\leq r}.
\end{equation}
Let $B_\ell\in\mathscr{B}_r$, $\ell=1,2$, we denote by $\para{\lambda_{\ell,k}, \phi_{\ell,k}}$, $k\geq 1$, the eigenvalues and normalized eigenfunctions of the operator $H_{B_\ell}$.

For $\ell= 1$ or $\ell= 2$, let 
\begin{equation}\label{1.15}
\psi_{\ell,k}=\para{\p_\nu+iA_\ell(\nu)}\phi_{\ell,k},\quad k\geq 1.
\end{equation}
 
At this point we remark that when $A_1=A_2$ it is clear that $H_{B_1}-H_{B_2}=q_1-q_2$ whence by the min-max principle, 
$$\sup_{k\geq1}|\lambda_{1,k}-\lambda_{2,k}|\leq \norm{q_1-q_2}_{L^\infty(M)}<\infty.$$
Assume now that $A_1\neq A_2$ and $\delta A_1=\delta A_2$. Then we have
$$H_{B_1}-H_{B_2}=-2i(A_1-A_2)\nabla+|A_1|^2-|A_2|^2+q_1-q_2.$$
Thus, $H_{B_1}-H_{B_2}\notin \mathcal B(L^2(M))$. Therefore, we can reasonably expect that
$$\sup_{k\geq1}|\lambda_{1,k}-\lambda_{2,k}|=+\infty.$$

Keeping in mind this property and the obstruction described in Section 1.2, it seems natural to expect the recovery of the solenoidal part of the magnetic potential from a rate of growth of the eigenvalues. Our first result give a positive answer to this issue together with the recovery of the electric potential.

\begin{theorem}\label{t1} 
Assume that $\M$ is simple. Let $B_\ell=(A_\ell,q_\ell)\in\mathscr{B}_r$, $\ell=1,2$, such that 
\bel{t3e} \p_x^\alpha A_1(x)= \p_x^\alpha A_2(x),\quad x\in\p M,\; |\alpha|\leq1.\ee 
Furthermore, assume that there exists $t\in [0,1/2)$ so that
\begin{equation}\label{1.16} 
\sup_{k\ge 1} k^{-t/n} \abs {\lambda_{1,k}-\lambda_{2,k}}+ \sum_{k\ge 1}k^{-2t/n}\|\psi_{1,k}-\psi_{2,k}\|_{L^2(\p \M)}^2<\infty . 
\end{equation}
Then $A^s_1=A^s_2$. Moreover, under the additional conditions
\begin{equation}\label{1.18} 
\lim_{k\to+\infty}\abs{\lambda_{1,k}-\lambda_{2,k}}=0,\quad  \textrm{and}\;\;  \sum_{k\ge 1}\norm{\psi_{1,k}-\psi_{2,k}}_{L^2(\p \M)}^2<\infty ,
\end{equation}
we have  $q_1=q_2$.
\end{theorem}

In the spirit of \cite{CS,KKS}, we consider also the stability issue for this problem stated as follows.
\begin{theorem}\label{t2} 
Assume that $\M$ is simple. Let $B_\ell=(A_\ell,q_\ell)\in\mathscr{B}_r$, $\ell=1,2$, such that
 $A_1$ and $A_2$ satisfies \eqref{t3e} and $q_1-q_2\in H^1_0(\M)$ satisfies 
 \[\norm{q_1-q_2}_{H^1_0(\M)}\leq r.
 \]
Furthermore, assume that there exists $t\in(0,1/2)$ so that
\begin{equation}\label{1.19} 
\sup_{k\ge 1}k^{-t/n}\abs{\lambda_{1,k}-\lambda_{2,k}}+ \sum_{k\ge 1}\|\psi_{1,k}-\psi_{2,k}\|_{L^2(\p \M)}^2<\infty .\end{equation}
Then $A^s_1=A^s_2$ and  
\begin{equation}\label{1.20} 
\norm{q_1-q_2}_{L^2(\M)}\leq C\left(\limsup_{k\to\infty}\abs{\lambda_{1,k}-\lambda_{2,k}}\right)^{\frac{1}{2}}<\infty,
\end{equation}
the constant $C$ only depends on $r$ and $\M$.
\end{theorem}

To our best knowledge Theorems \ref{t1} and \ref{t2} are the first results dealing with inverse spectral problems for  Schr\"odinger operators, with non-constant leading coefficients, from asymptotic knowledge  of boundary spectral data similar to the one considered by \cite{CS,KKS}. Note also that Theorem \ref{t2} seems to be the first stability result of recovering the electric potential  from  partial boundary spectral data in such general context (the only other similar results can be found in \cite{BCY,CS,KKS} where stable recovery of Schr\"odinger operators on a bounded domain, with an Euclidean metric and without magnetic potential, have been considered).

 We recall that multi-dimensional Borg-Levinson type theorems for magnetic Schr\"odinger operators have been already considered in \cite{KK,Ki1,Ser}. Among them, only \cite{Ki1} considered the uniqueness issue from boundary spectral data similar to \eqref{1.18}. The results in the present work can be seen as an improvement of that in \cite{Ki1} in four directions. First of all, we prove for the first time the extension of such results to a general simple Riemanian manifold  by proving the connection between our problem and the injectivity of the so called geodesic ray transform borrowed from \cite{AR,FSU,Sh,SU}. In addition, by using some results of  \cite{SU}, we establish  stability estimates for this problem where \cite{Ki1} treated only the uniqueness. In contrast to \cite{Ki1}, we do not require the knowledge of the magnetic potentials on the neighborhood of the boundary. This condition is relaxed, by considering only some knowledge of the magnetic potentials at the boundary given by \eqref{t3e}. Finally, we show, for what seems to be the first time,  that even a rate of growth of the difference of eigenvalues like \eqref{1.16}, \eqref{1.19} can determine the magnetic potential appearing in a magnetic Schr\"odinger operator.

The main ingredient in our analysis is a suitable representation formula that involves the magnetic potential $A$ and the electric potential $q$ in terms of the Dirichlet-to-Neumann map associated to the equations $\mathcal H_{B}u-\lambda u=0$, for a well chosen set of complex $\lambda$'s. In \cite{I,Ki1}, the authors considered such a representation for a bounded domain with flat  metric. Using a construction inspired by  \cite{Be,BD,FKSU,FKSaU, FKLS,SU2} we show how one can extend such  approach to more general  manifolds. Note that this construction differs from the one considered by \cite{FKSU,FKSaU, FKLS} for recovering the magnetic Schr\"odinger operators from boundary measurements. Actually, our results hold for a general simple manifold even in the case $n\geq3$, whereas the determination of Schr\"odinger operators from boundary measurements in the same context is still an open problem (see \cite{FKSaU,FKLS}).

In this paper we treat also the problem of determining the Neumann realization of magnetic Schr\"odinger operator. For simplicity and in order to avoid any confusion between the results for the different operators, we give the statement of the result for  the Neumann realization of magnetic Schr\"odinger operator in Theorem \ref{t3} of Section 6. the result of Theorem \ref{t3} is stated with an optimal growth   of the difference of eigenvalues (see the discussion just after Theorem \ref{t3}).

We believe that following the idea of \cite{BKS,Ki1,KU,Sa}, one can relax the regularity condition imposed to the magnetic potentials as well as condition \eqref{t3e}. This approach requires the construction of anzats depending on an approximation of the magnetic potential instead of the magnetic potential itself. In order to avoid the inadequate expense of the size of the paper, we do not consider this issue.

\subsection{Outline}
The outline of the paper is as follows. We review  in Section 2  the geodesic ray transform for $1$-one forms and functions on a manifold. Section 3 is devoted to an asymptotic spectral analysis. We construct  in Section 4 geometrical optics solutions for magnetic Schr\"odinger equations. We particularly focus our attention on the solvability of the eikonal and the transport equations which are essential in the construction of geometric optic solutions. Additionally,  we provide a representation formula. The proof of Theorems \ref{t1} and \ref{t2} are given in Section 5. The Neumann case is briefly discussed in Section 6.  Finally, we prove some uniform estimates related to the Weyl's formula for the magnetic Schr\"odinger operator in  appendix A.  
\section{A short review on the geodesic ray transform on a simple manifold}
\setcounter{equation}{0}
We collect in this section some known results on the geodesic ray transform for  functions  and $1$-forms on a smooth simple Riemannian manifold $(\M,\g)$. These results will be used later in this text.

Denote by $\dive X$ the divergence of a vector field $X\in H^1(\M,T\M)$ on $\M$, i.e. in local coordinates (\cite[page 42]{KKL}),
\begin{equation}\label{2.1}
\dive X=\frac{1}{\sqrt{\abs{\g}}}\sum_{i=1}^n\p_i\para{\sqrt{\abs{\g}}\,X^i},\quad X=\sum_{i=1}^nX^i\p_i.
\end{equation}
Using the inner product of a $1$-form, we can define the coderivative operator $\delta$ as the adjoint of the exterior derivative via the relation
\begin{equation}\label{2.2}
\para{\delta A,v}=\para{A,dv},\quad A\in\mathcal{C}^\infty(M,T^*M),\,v\in \mathcal{C}^\infty(\M).
\end{equation}
Then $\delta A$ is related to the divergence of vector fields by $\delta A=\dive(A^\sharp)$, where the divergence is given by \eqref{2.1}. 
If $X\in H^1(\M,T\M)$ the divergence formula reads
\begin{equation}\label{2.3}
\int_\M\dive X \dv=\int_{\p \M}\seq{X,\nu} \ds.
\end{equation}
For $f\in H^1(\M)$, we have the following Green formula 
\begin{equation}\label{2.4}
\int_\M\dive X\,f\dv=-\int_\M\seq{X,\nabla f} \dv+\int_{\p \M}\seq{X,\nu} f\ds.
\end{equation}
Therefore, for $u,w\in H^2(\M)$, the following identity holds
\begin{align}
\int_\M\Delta_Au \overline{w} \dv &=-\int_\M\seq{\nabla_A u,\overline{\nabla_A w} } \dv+\int_{\p \M}(\p_\nu u+iA(\nu)u) \overline{w} \ds \label{2.5}
\\
&= \int_\M u\overline{\Delta_Aw} \dv \nonumber
\\
&\hskip 1cm +\int_{\p \M}\para{(\p_\nu u+iA(\nu)u) \overline{w}-u(\overline{\p_\nu w+iA(\nu)w})} \ds,\nonumber
\end{align}
where $\nabla_Au=\nabla u+iuA^\sharp$.
For $x\in \M$ and $\theta\in T_x\M$, denote by $\gamma_{x,\theta}$ the unique geodesic starting from $x$ and directed by $\theta$.  

Recall that the sphere bundle and co-sphere bundle of $\M$ are respectively given by
\begin{align*}
S\M=\set{(x,\theta)\in T\M;\,\abs{\theta}=1}, \quad
S^*\M=\set{(x,p)\in T^*\M;\,\abs{p}=1},
\end{align*}

The exponential map $\exp_x:T_x\M\To \M$ is defined as follows
\begin{equation}\label{2.6}
\exp_x(v)=\gamma_{x,\theta}(\abs{v}),\quad \theta=\frac{v\,\,}{\abs{v}}.
\end{equation}


We assume in the rest of this section that $M$ is simple and we point out  that  any arbitrary pair of
points in $M$ can be joined by an unique geodesic of finite length.

Given $(x,\theta)\in S\M$ and denote by $\gamma_{x,\theta}$ the unique geodesic $\gamma_{x,\theta}$ satisfying the initial conditions $\gamma_{x,\theta}(0) = x$ and $\dot{\gamma}_{x,\theta}(0) = \theta$, which is defined on the maximal interval $[\ell_-(x,\theta),\ell_+(x,\theta)]$, with $\gamma_{x,\theta}(\ell_\pm(x,\theta))\in\p\M$. Define the geodesic flow $\varphi_t$ by
\begin{equation}\label{2.7}
\varphi_t:S\M\to S\M,\quad \varphi_t(x,\theta)=(\gamma_{x,\theta}(t),\dot{\gamma}_{x,\theta}(t)),\quad t\in [\ell_-(x,\theta),\ell_+(x,\theta)],
\end{equation}
and observe that $\varphi_t\circ\varphi_s=\varphi_{t+s}$.

Introduce now the submanifolds of inner and outer vectors of $S\M$
\begin{equation}\label{2.8}
\p_{\pm}S\M =\set{(x,\theta)\in S\M,\, x \in \p \M,\, \pm\seq{\theta,\nu(x)}< 0},
\end{equation}
where $\nu$ is the unit outer normal vector field on $\p\M$. 

Note that $\p_+ S\M$ and $\p_-S\M$ are compact manifolds with the same boundary $S(\p \M)$ and \[ \p S\M = \p_+ S\M \cup S \p \M \cup \p_- S\M. \]


It is straightforward to check that $\ell_\pm:S\M\to\R$ satisfy
$$
\ell_-(x,\theta)\leq 0,\quad \ell_+(x,\theta)\geq 0,
$$
$$
\ell_+(x,\theta)=-\ell_-(x,-\theta),
$$
$$
\ell_-(x,\theta)=0,\quad (x,\theta)\in\p_+S\M,
$$
$$
\ell_-(\varphi_t(x,\theta))=\ell_-(x,\theta)-t,\quad \ell_+(\varphi_t(x,\theta))=\ell_+(x,\theta)+t.
$$

To each $1$-form $A\in\mathcal{C}^\infty(\M,T^*\M)$, with $A=a_jdx^j$, associate the smooth symbol $\sigma_A\in\mathcal{C}^\infty(S\M)$ given by
\begin{equation}\label{2.9}
\sigma_A(x,\theta)=\sum_{j=1}^na_j(x)\theta^j=\seq{A^\sharp(x),\theta},\quad (x,\theta)\in S\M.
\end{equation}

Recall that the Riemannian scalar product on $T_x\M$ induces the volume form on $S_x\M$ given by
$$
\dd \omega_x(\theta)=\sqrt{\abs{\g}} \, \sum_{k=1}^n(-1)^k\theta^k \dd \theta^1\wedge\cdots\wedge \widehat{\dd \theta^k}\wedge\cdots\wedge \dd \theta^n.
$$
As usual, the notation $\, \widehat{\cdot} \,$ means that the corresponding factor has been dropped.

We also consider the volume form $\dvv$ on the manifold $S\M$ defined as follows
$$
\dvv (x,\theta)=\dd\omega_x(\theta)\wedge \dv,
$$
where $\dv$ is the Riemannnian volume form on $\M$. 

By Liouville's theorem, the form $\dvv$ is preserved by the geodesic flow. The
corresponding volume form on the boundary $\p S\M =\set{(x,\theta)\in S\M,\, x\in\p \M}$ is given
by
$$
\dss=\dd\omega_x(\theta) \wedge \ds,
$$
where $\ds$ is the volume form of $\p \M$.


Santal\'o's formula will be useful in the sequel:
\begin{equation}\label{2.10}
\int_{S\M} F(x,\theta) \dvv (x,\theta)=\int_{\p_+S\M}\para{\int_0^{\ell_+(x,\theta)} F\para{\varphi_t(x,\theta)}dt}\mu(x,\theta)\dss ,
\end{equation}
for any $F\in\mathcal{C}(S\M)$.

Set $\mu(x,\theta)=\abs{\seq{\theta,\nu(x)}}$. For the sake of simplicity $L^2\left(\p_+S\M ,\mu(x,\theta)\dss \right)$ is denoted by $L^2_\mu(\p_+S\M)$. 

Note that $L^2_\mu(\p_+S\M)$ is a Hilbert space when it is endowed with the scalar product
\begin{equation}\label{2.11}
\para{u,v}_\mu=\int_{\p_+S\M}u(x,\theta) \overline{v}(x,\theta) \mu(x,\theta)\dss.
\end{equation}

Until the end of this section, we assume that $\M$ is simple.

\subsection{Geodesic ray transform of $1$-forms}

The ray transform of $1$-forms on $\M$ is defined as the linear operator
$$
\I_1:\mathcal{C}^\infty(\M, T^*\M)\To \mathcal{C}^\infty(\p_+S\M)
$$
acting as follows
$$
\I_1 (A)(x,\theta)=\int_{\gamma_{x,\theta}}A=\sum_{j=1}^n\int_0^{\ell_+(x,\theta)} a_j(\gamma_{x,\theta}(t))\dot{\gamma}^j_{x,\theta}(t)dt=\int_0^{\ell_+(x,\theta)}\sigma_A(\varphi_t(x,\theta))dt.
$$
It is easy to check that $\I_1(d\varphi)=0$ for any $\varphi\in\mathcal{C}^\infty(\M )$ satisfying $\varphi_{|\p\M}=0$. On the other hand, it is known that $\I_1$ is injective on the space of solenoidal $1$-forms satisfying $\delta A=0$. Therefore, if $A\in H^1(\M,T^*\M)$ is so that $\I_1(A)=0$, then $A^s=0$. Whence, there exists $\varphi \in H^1_0(M) \cap H^2(M)$ such that  $A=d\varphi$.
As a consequence of this observation, we have
\begin{equation}\label{2.12}
\abs{\I_1(A)(x,\theta)}=\abs{\I_1(A^s)(x,\theta)}\leq C\norm{A^s}_{\mathcal{C}^0}, \quad A\in \mathcal{C}^0(\M,T^*\M).
\end{equation}
With reference to \cite{Sh}, we recall  that $\I_1^*:L^2_\mu(\p_+S\M)\To L^2(\M,T^*\M)$ is given by
\begin{equation}\label{2.14}
\para{\I_1^*\Psi(x)}_j=\int_{S_x\M}\theta^j\widecheck{\Psi}(x,\theta)\, \dd\omega_x(\theta).
\end{equation}

Here $\widecheck{\Psi}$ is the extension of  $\Psi$ from $\p_+S\M$ to $S\M$, which is constant on every orbit of the geodesic flow. That is
$$
\widecheck{\Psi}(x,\theta)=\Psi\big(\gamma_{x,\theta}(\ell_-(x,\theta)),\dot{\gamma}_{x,\theta}(\ell_{-}(x,\theta))\big)=\Psi(\Phi_{\ell_-(x,\theta)}(x,\theta)),\quad (x,\theta)\in S\M.
$$

One can check \cite{Sh} that $\I_1$ has  a bounded extension, still denoted by $\I_1$, 
$$
\I_1:H^k(\M,T^*\M)\To H^k (\p_+S\M).
$$
We complete this subsection by results borrowed from \cite{SU}. We extend $(M,\g)$ to a smooth Riemannian manifold  $(M_1,\g)$ such that $\M_1^{\textrm{int}}\supset \M$
and we consider the normal operator $N_1=\I_1^*\I_1$.  
Then there exist
$C_1>0, C_2>0$ such that
\begin{equation}\label{2.15}
C_1\norm{A^s}_{L^2(\M)}\leq\norm{N_1(A)}_{H^1(\M_1)}\leq C_2\norm{A^s}_{L^2(\M)},
\end{equation}
for any $A\in L^2(\M,T^*\M)$. If $\mathcal{O}$ is an open set of $\M_{1}$, $N_1$ is an
elliptic  pseudo-differential operator of order $-1$ on $\mathcal{O}$  having as  principal symbol $\varrho(x,\xi)=(\varrho_{jk}(x,\xi))_{1\leq j,k\leq n}$, where
$$
\varrho_{j,k}(x,\xi)=\frac{c_n}{\abs{\xi}}\para{\g_{jk}-\frac{\xi_j\xi_k}{\abs{\xi}^2}}.
$$
Therefore, for each integer $k\geq 0$, there exists a constant $C_k>0$ such that, for any $A\in H^k(\M,T^*\M)$ compactly supported in $\mathcal{O}$, we have
\begin{equation}\label{2.16}
\norm{N_1(A)}_{H^{k+1}(\M_{1})}\leq C_k\norm{A^s}_{H^k(\mathcal{O})}.
\end{equation}
\subsection{Geodesic ray transform of functions}
Following \cite[Lemma 4.1.1]{Sh}, the ray transform of functions is the linear operator
\begin{equation}\label{2.17}
\I_0:\mathcal{C}^\infty(\M)\To \mathcal{C}^\infty(\p_+S\M)
\end{equation}
acting as follows
\begin{equation}\label{2.18}
\I_0 f(x,\theta)=\int_0^{\ell_+(x,\theta)}f(\gamma_{x,\theta}(t))\, \dd t.
\end{equation}
Similarly to $\I_1$, $\I_0$ has an extension, still denoted by $\I_0$:
\begin{equation}\label{2.19}
\I_0:H^k(\M)\To H^k(\p_+S\M)
\end{equation}
for every integer $k\geq 0$. We refer to  \cite[Theorem 4.2.1]{Sh} for details.
\medskip

Considering $\I_0$ as a bounded operator from $L^2(\M)$ into
$L^2_\mu(\p_+S\M)$, we can compute its adjoint $\I_0^*:L^2_\mu(\p_+S\M)\to L^2(\M)$
\begin{equation}\label{2.20}
\I_0^*\Psi(x)=\int_{S_x\M}\widecheck{\Psi}(x,\theta)\, \dd\omega_x(\theta),
\end{equation}
where $\widecheck{\Psi}$ is the extension of $\Psi$ from $\p_+S\M$ to $S\M$ which is constant on every orbit of the geodesic flow:
$$
\widecheck{\Psi}(x,\theta)=\Psi(\gamma_{x,\theta}(\ell_+(x,\theta))).
$$
Let $\M_1$ be a simple manifold so that $\M_1^{\textrm{int}}\supset \M$ and consider the normal operator $N_0=\I_0^*\I_0$. Then there exist two constants
$C_1>0, C_2>0$ such that
\begin{equation}\label{2.21}
C_1\norm{f}_{L^2(\M)}\leq\norm{N_0(f)}_{H^1(\M_1)}\leq C_2\norm{f}_{L^2(\M)}
\end{equation}
for any $f\in L^2(\M)$, see \cite{SU}. 

If $\mathcal{O}$ is an open set of $\M_{1}$,  $N_0$ is an
elliptic  pseudo-differential operator of order $-1$ on $\Omega$, whose principal symbol is a multiple of $\abs{\xi}^{-1}$, see \cite{SU}.
Therefore there exists a constant $C_k>0$ such that, for all $f\in H^k(\mathcal{O})$ compactly supported in $\mathcal{O}$,
\begin{equation}\label{2.22}
\norm{N_0(f)}_{H^{k+1}(\M_{1})}\leq C_k\norm{f}_{H^k(\mathcal{O})}.
\end{equation}
\section{Asymptotic spectral analysis}\label{sec:Prelim}
\setcounter{equation}{0}
We fix in all of this section $B_\ell=(A_\ell,q_\ell)\in \mathscr{B}_r$, $\ell=1,2$, satisfying the assumptions of Theorem \ref{t1}. As in Section 1, $H_{B_\ell}$, $\ell= 1,2$, is the operator defined by \eqref{1.2} and \eqref{1.3} when $B=B_\ell$. Furthermore, for $\lambda \in \rho (H_{B_\ell})$, denote by $R_{B_\ell}(\lambda)$ the resolvent of $H_{B_\ell} $ and, for $s\in [0,1/2)$, recall the following classical resolvent estimate 
\begin{equation}\label{3.1}
\norm{R_{B_\ell}(\lambda)}_{\mathscr{L}\para{L^2(\M);H^{2s}(\M)}}\leq \frac{C_s}{\abs{\Im \lambda}^{1-s}},\quad\ell=1,2.
\end{equation}
For $f \in H^{3/2}(\p \M)$ and $\lambda \in \rho (H_{B_\ell})$, $\ell=1,2$, consider the Dirichlet problem
\begin{equation}\label{3.2}
\left\{ 
\begin{array}{ll} 
(\mathcal{H}_{B_\ell} -\lambda )u=0  &\textrm{in}\,\, \M ,\cr
u =f  &\textrm{on}\,\, \p \M.
\end{array}
\right.
\end{equation}

 Let $\kappa$ be a boundary defining function, that is a smooth function $\kappa : \bar{\M} \to \R_+$ such that
\begin{itemize}
     \item $\kappa(x)>0$ for all $x \in \M^{\rm int}$,
     \item $\kappa|_{\p\M}=0$ and $d\kappa|_{\p M} \neq 0$.
\end{itemize}
We recall that one can construct such a function by combining local coordinates with boundary distance functions or by considering the first eigenvalue of the Dirichlet Laplacian.
We can now state the following result.
\begin{lemma}\label{L.3.1}
If $f \in H^{3/2}(\p \M)$ and $\lambda \in \rho (H_{B_\ell})$, then the BVP \eqref{3.2} has a unique solution $u_\ell(\lambda ) =u_\ell ^f(\lambda ) \in H^2(\M)$ given by the series
\begin{equation}\label{3.3}
u_\ell(\lambda )= \sum_{k \geq 1} \frac{\seq{ f,\psi_{\ell,k}}}{\lambda - \lambda_{\ell,k}}\,  \phi_{\ell,k},
\end{equation}
the convergence takes place in $H^1(\M)$. Moreover, for any neighborhood $\mathcal{V}$ of $\p \M$ in  $\M$,  we have
\begin{equation}\label{3.4}
\lim_{\lambda\to-\infty} \para{\|u_{\ell}(\lambda)\|_{L^2(\M)} +\norm{\kappa d u_{\ell}(\lambda)}_{L^2(\M )}}= 0 .
\end{equation}
\end{lemma}
\begin{proof}
The proof of \eqref{3.3} and
$$
\lim_{\lambda\to-\infty} \|u_{\ell}(\lambda)\|^2_{L^2(\M)} = 0
$$
is quite similar to that of \cite[Lemma 2.1]{Ki1}. \\
The proof of \eqref{3.4} is then completed by establishing the following Caccioppoli's type inequality, where $\lambda<0$:
\begin{equation}\label{3.5}
 \norm{\kappa d u_\ell(\lambda )}_{L^2(\M)}\leq C\norm{u_{\ell}(\lambda)}_{L^2(\M)},
 \end{equation}
 the constant $C$ only depends on $r$ and $\M$.
 
For the sake of simplicity, we omit the subscript  $\ell$ in $u_\ell(\lambda)$ and $B_\ell$.  Multiplying the first equation of \eqref{3.2} by $\kappa^2\overline{u}(\lambda)$, using the fact that $\kappa_{|\partial \M}=0$ and applying Green's formula, we obtain
\begin{align}
0&=-\int_\M \Delta_A u(\lambda) \kappa^2\overline{u}(\lambda)\dv+\int_\M(q-\lambda)\kappa^2\abs{u(\lambda)}^2\dv \label{3.6}
\\
&= \int_\M |\kappa du(\lambda )|^2\dv+2\int_\M \seq{\kappa  du(\lambda ),  \overline{u}(\lambda) d\kappa}\dv \nonumber
\\
&\quad
+2 \Im \int_\M \seq{\kappa \overline{u}(\lambda )A,\kappa du(\lambda)}\dv +\int_\M\big(2i \langle A,\kappa d\kappa\rangle+(|A|^2+q-\lambda)\kappa^2\big)\abs{u(\lambda )}^2\dv.\nonumber
\end{align}
An application of Cauchy-Schwarz's inequality yields
\begin{align*} 
\norm{\kappa d u(\lambda )}_{L^2(\M)}^2-\lambda\norm{\kappa u(\lambda )}^2_{L^2(\M)} &\leq C\norm{u(\lambda ) }_{L^2(\M)}\norm{\kappa d u(\lambda )}_{L^2(\M)}+C\norm{u(\lambda )}^2_{L^2(\M)}
\\
&\leq C'\norm{u(\lambda ) }^2_{L^2(\M)}+\frac{1}{2}\norm{\kappa du(\lambda )}_{L^2(\M)}^2.
\end{align*}
Then, it follows
\begin{equation}\label{3.7}
\frac{1}{2}\norm{\kappa d u(\lambda )}_{L^2(\M)}^2-\lambda\norm{\kappa u(\lambda )}^2_{L^2(\M)}\leq C\norm{u(\lambda )}^2_{L^2(\M)}
\end{equation}
and since $-\lambda >0$
\begin{equation}\label{3.8}
\norm{\kappa d u(\lambda )}_{L^2(\M)}^2\leq C\norm{ u(\lambda)}^2_{L^2(\M)},
\end{equation}
implying Caccioppoli's inequality \eqref{3.5}.
\end{proof}
\begin{lemma}\label{L.3.2} 
Let  $f \in H^{3/2}(\p \M)$ and,  for $\mu \in \rho (H_{B_1})\cap \rho (H_{B_2})$, set 
\[w_{1,2}(\mu )=u_{1}(\mu )-u_{2}(\mu ) \in H^2(\M),\]  where $u_{\ell}(\mu)$ is the corresponding solution to \eqref{3.2} with $B_\ell$ and $\lambda$ are substituted by $B_\ell$ and $\mu$. Then we have that  $w_{1,2}(\mu )$ converges to $0$ in $H^{2}(\M)$ as $\mu\to-\infty$. In particular, $\partial_\nu w_{1,2}(\mu ) \to 0$ in $L^2(\p \M)$ as $\mu \to -\infty$.
\end{lemma}
\begin{proof}
For the sake of simplicity, we use in this proof $w(\mu)$ instead of $w_{1,2}(\mu )$. Since the trace map $v\mapsto \partial_\nu v$ is  continuous from $H^{2}(\M)$ into $L^2(\p \M)$, it is enough to show that $\|w(\mu )\|_{H^{2}(\M)} \to 0$ when $\mu \to -\infty$.  Let $\mu<\mu_\ast<-2\norm{q}_\infty$, for some fixed $\mu _\ast<0$. It is straightforward to check that $w(\mu )$ is the solution of the boundary value problem
\begin{equation}\label{3.9}
\left\{ 
\begin{array}{ll} 
\left(\mathcal{H}_{B_1} - \mu \right)w(\mu )=h(\mu)  & \textrm{in}\,\, \M,
\\ 
w(\mu )= 0 &\textrm{on}\,\,  \p \M.
\end{array}
\right.
\end{equation}
Here $h(\mu)$ is given by
\begin{equation}\label{3.10}
h(\mu )=-2i\seq{A_2-A_1, du_{2}(\mu)}+\left(V_2-V_1\right)u_{2}(\mu )
\end{equation}
with 
$$
V_j=-i\delta A_j+\abs{A_j}^2+q_j, \quad j=1,2.
$$ 
Multiplying the first equation of \eqref{3.9} by $\overline{w}(\mu)$, we apply Green's formula \eqref{2.5} in order to obtain
\begin{align*}
\int_\M h(\mu)\overline{w}(\mu)\dv&=\int_\M\h_{B_1}w(\mu)\overline{w}(\mu)\dv-\int_\M\mu\abs{w(\mu)}^2\dv\cr
&=\int_\M\abs{\nabla_{A_1} w}^2\dv+\int_\M(q-\mu)\abs{w}^2\dv.
\end{align*}
We deduce that, for $-\mu$ sufficiently large,
$$
(-\norm{q}_\infty-\frac{\mu}{2})\norm{w(\mu)}^2_{L^2(\M)}+\frac{\abs{\mu}}{4}\norm{w(\mu)}^2_{L^2(\M)}\leq C\norm{h(\mu)}^2_{L^2(\M)},
$$
for some positive constant $C$, not dependent  on $\mu$, and then we conclude that
\begin{equation}\label{3.11}
\abs{\mu}\norm{w(\mu)}^2_{L^2(\M)}\leq C\norm{h(\mu)}^2_{L^2(\M)}.
\end{equation}
Moreover we have
\begin{equation}\label{3.12}
\left\{ 
\begin{array}{ll} 
\left(H_{B_1} - \mu_* \right)w(\mu )=h(\mu)+(\mu-\mu_*)w(\mu)  & \textrm{in}\,\, \M,
\\ 
w(\mu )= 0 &\textrm{on}\,\,  \p \M.
\end{array}
\right.
\end{equation}
Using that  $(H_{B_1} - \mu^\ast )^{-1}$ is an isomorphism from $L^2(\M)$ onto $H^2(\M)$, there exists a constant $C$, 
depending on $\M$ and $B_1$, so that
\begin{eqnarray}\label{3.13}
\norm{w(\mu )}_{H^2(\M)}&\leq & C\norm{h(\mu)+(\mu-\mu_*)w(\mu))}_{L^2(\M)}\cr
&\leq & C\para{\norm{h(\mu)}_{L^2(\M)}+\abs{\mu-\mu_*}\norm{w(\mu)}_{L^2(\M)}}\cr
&\leq &  C\para{\norm{h(\mu)}_{L^2(\M)}+2\abs{\mu}\norm{w(\mu)}_{L^2(\M)}},
\end{eqnarray}
where the positive constant $C$ is not dependent  on $\mu$.\\
Using now the  estimate \eqref{3.11},  we obtain
\begin{equation}\label{3.14}
\norm{w(\mu )}_{H^2(\M)}\leq 4C\norm{h(\mu)}_{L^2(\M)}.
\end{equation}
On the other hand,  in view of \eqref{t3e} there exists $C>0$  such that 
\begin{equation}\label{tata3.15}|A_1(x)-A_2(x)|\leq C\kappa(x),\quad x\in \M.\end{equation}
Applying \eqref{tata3.15}, we obtain
\begin{equation}\label{3.15}
\norm{h(\mu )}_{L^2(\M)}\leq C''\left(\norm{\kappa d u_{2}(\mu )}_{L^2(\M)}+\norm{ u_{2}(\mu )}_{L^2(\M)}\right)
\end{equation}
for some constant $C''$ independent of $\mu$. Then, according to \eqref{3.4} in Lemma \ref{L.3.1}, we get
\begin{equation}\label{3.16}
\limsup_{\mu\to-\infty}\norm{h(\mu )}_{L^2(\M)}=0,
\end{equation}
entailing by \eqref{3.14}
\begin{equation}\label{3.17}
\limsup_{\mu\to-\infty}\norm{w(\mu )}_{H^2(\M)}=0.
\end{equation}
This completes the proof of the lemma.
\end{proof}
The following lemma will be useful in the sequel. We omit its proof since it is quite similar to that in \cite{KKS,Ki1}.
\begin{lemma}\label{L.3.3}
Let $f\in H^{3/2}(\p\M)$ and, for $\mu,\lambda\in\rho(H_{B_\ell})$, set $w_\ell(\lambda ,\mu)=  u_\ell(\lambda ) - u_\ell(\mu)$, where $u_\ell(\mu)$ is the solution of \eqref{3.2} when $\lambda$ is substituted by $\mu$. Then we have  
\begin{equation}\label{3.18}
\para{\partial_\nu+iA_\ell(\nu)} w_\ell(\lambda ,\mu)  = \sum_{k\geq1} 
\frac{(\mu - \lambda)\seq{ f,\psi_{\ell,k}}}{(\lambda - \lambda_{\ell,k})(\mu - \lambda_{\ell,k})}\, \psi_{\ell,k},
\end{equation}
the convergence takes place in $H^{1/2}(\p \M)$.
\end{lemma}
\section{Isozaki's representation formula}
\setcounter{equation}{0}
In the present section we provide a version of Isozaki's approach \cite{I}, based on the so-called \textit{Born approximation} method. The usual anzats used to solve the problem of determining the coefficients of a magnetic Laplace-Beltrami operator, from the corresponding Dirichlet-to-Neumann map will be useful in our analysis. 
Let us describe briefly this method.

In all of this section $B_\ell=(A_\ell,q_\ell)\in\mathscr{B}_r$, $\ell=1,2$, with $A_\ell$ satisfying \eqref{t3e}. We extend the covector $A_1$ to a $W^{2,\infty}$ covector  on $\M_1$ supported in the interior of $\M_1$ and still denoted by $A_1$. Then, we consider the extension of $A_2$ to $\M_1$, still denoted by $A_2$, defined by 
\bel{eqi} A_1(x)=A_2(x),\quad x\in \M_1\setminus \M.\ee
Then, \eqref{t3e} implies that $A_2\in W^{2,\infty}(\M_1;T^*\M_1)$.
We fix also  $A=A_1-A_2$.
\subsection{Representation formula}
If $u_\ell(\lambda )$, $\lambda \in \rho (H_{B_1})\cap \rho (H_{B_2})$, is the solution of \eqref{3.2} when $B=B_\ell$, define  the Dirichlet-to-Neumann map by
\begin{equation}\label{4.1}
\Lambda_{B_\ell}(\lambda ): f\in H^{{3/2}}(\partial \M) \mapsto\para{\partial_\nu +iA_\ell(\nu)}u_{\ell}(\lambda ){_{|\p \M}},\quad \ell=1,2.
\end{equation}
We fix $\psi\in\mathcal{C}^2(\M)$  a function satisfying  the eikonal equation
\begin{equation}\label{4.2}
\abs{d\psi}^2=\sum_{i,j=1}^n\g^{ij}\frac{\p\psi}{\p x_i}\frac{\p\psi}{\p
x_j}=1.
\end{equation}
We set also two functions $\alpha_\ell\in H^2(\M)$ solving the transport equations
\begin{equation}\label{4.3}
\seq{d\psi,d\alpha_\ell}+\frac{1}{2} (\Delta \psi)\alpha_\ell=0,\quad \ell=1,2.
\end{equation}
This function will be given in Section 4.2.
Consider also two functions $\beta_{A_\ell}\in H^2(\M) $, $\ell=1,2$, solutions of the transport equations
\begin{equation}\label{4.4}
\seq{d\psi,d\beta_{A_\ell}}+i \seq{A_\ell ,d\psi}\beta_{A_\ell}=0,\quad \forall \, x\in\M,\quad\ell=1,2.
\end{equation}
Henceforth $\tau>1$ and $\lambda_\tau =\tau+i $. Let
\begin{equation}\label{4.5} 
\begin{array}{lll}
\varphi_{1,\tau} ^\ast(x)=e^{i\lambda_\tau\psi(x)}\alpha_1\beta_{A_1}(x):=e^{i\lambda_\tau\psi(x)}\bb_1 (x),\cr \ww(x)=e^{i\overline{\lambda_\tau}\psi(x)}\alpha_2\beta_{A_2}(x):=e^{i\overline{\lambda_\tau}\psi(x)}\bb_2 (x),
\end{array}
\end{equation}
where, for $\ell=1,2$, $\alpha_\ell$ is a solution of \eqref{4.3} and $\beta_{A_\ell}$ is a solution of \eqref{4.4}.

Define
\begin{equation}\label{4.6} 
S_{B_\ell}(\tau)=\seq{\Lambda_{B_\ell}(\lambda_\tau ^2) \w,\ww}=\int_{\p \M}\Lambda_{B_\ell}(\lambda_\tau ^2) \w\overline{\ww} \ds,\quad  \ell=1,2.
\end{equation}
\begin{lemma}\label{L.4.1} 
We have
\begin{multline}\label{4.7}
S_{B_1}(\tau)=\int_{\p\M}\bb_1\para{\p_\nu\overline{\bb_2}-iA_1(\nu)\overline{\bb_2}-i\lambda_\tau\overline{\bb_2}\p_\nu\psi}\ds\cr
+\int_\M\bb_1\overline{\h_{B_1} (\bb_2)}\dv-2\lambda_\tau\int_\M\bb_1\overline{\bb_2}\seq{A,d\psi}\dv\cr
-\int_\M R_{B_1}(\lambda_\tau^2) \para{e^{i\lambda_\tau\psi}\h_{B_1}\para{\bb_1}}\para{e^{-i\lambda_\tau\psi}\overline{\h_{B_1} (\bb_2)}-2\lambda_\tau e^{-i\lambda_\tau\psi}\overline{\bb_2}\seq{A,d\psi}}\dv.
\end{multline}
and
\begin{multline}\label{4.8}
S_{B_2}(\tau)=\int_{\p\M}\bb_1\para{\p_\nu\overline{\bb_2}-iA_2(\nu)\overline{\bb_2}-i\lambda_\tau\overline{\bb_2}\p_\nu\psi}\ds+\int_M\bb_1\overline{\h_{B_2} (\bb_2)}\dv\cr
-\int_\M R_{B_2}(\lambda_\tau^2)\para{e^{i\lambda_\tau\psi} \para{\h_{B_2}\para{\bb_1}-2\lambda_\tau\seq{A,d\psi}\bb_1}}\para{e^{-i\lambda_\tau\psi}\overline{\h_{B_2} (\bb_2)}}\dv.
\end{multline}
Here  $R_{B_\ell}(\lambda_\tau^2)$ is the resolvent of $H_{B_\ell}$.
\end{lemma}
\begin{proof} 
Direct computations yield
\begin{align}
&\left(\mathcal{H}_{B_1} -\lambda_\tau ^2\right)\varphi_{1,\tau} ^\ast=e^{i\lambda_\tau\psi}\h_{B_1}\para{\bb_1} \label{4.9}
\\
&\qquad +e^{i\lambda_\tau\psi}\bigg(\lambda_\tau^2 \bb_1 \para{\abs{d\psi}^2-1} -2i\lambda_\tau \beta_{A_1}\para{\seq{d\psi,d\alpha_1}+\frac{\alpha_1}{2}\Delta\psi } \nonumber \\
&\qquad -2i\lambda_{\tau} \alpha_1 \para{\seq{d\psi,d\beta_{A_1}} +i \seq{A_1,d\psi}\beta_{A_1}} \bigg) \nonumber
\end{align}
Taking into account  \eqref{4.2} and \eqref{4.3}-\eqref{4.4}, with $\ell=1$, the right-hand side of \eqref{4.9} becomes
\begin{equation}\label{4.10}
\left(\mathcal{H}_{B_1} -\lambda_\tau ^2\right)\w=e^{i\lambda_\tau\psi}\h_{B_1}\para{\bb_1} \equiv e^{i\lambda_\tau\psi}k_1.
\end{equation}
Denote by $u_1$ the solution of the BVP
\[
\left\{ 
\begin{array}{ll} 
\left(\mathcal{H}_{B_1}-\lambda_\tau ^2\right) u_1=0\quad  &\textrm{in}\; \M ,
\\ 
u_1=\w &\textrm{on}\; \p \M.
\end{array}
\right.
\]
We  split $u_1$ into two terms, $u_1=\w+v_1$, where $v_1$ is the solution of the boundary value problem
\[
\left\{ 
\begin{array}{ll} 
\left(\mathcal{H}_{B_1}-\lambda_\tau ^2\right) v_1= -e^{i\lambda_\tau \psi} k_1\quad  &\textrm{in}\; \M,
\\ 
v_1=0 &\textrm{on}\; \p \M.
\end{array}
\right.
\]
Therefore
\begin{equation}\label{4.11}
u_1=\w-\left(H_{B_1}-\lambda_\tau ^2\right)^{-1}(e^{i\lambda_\tau\psi} k_1)=\w-R_{B_1}(\lambda_\tau^2) \para{e^{i\lambda_\tau\psi} k_1}.
\end{equation}
As
\begin{equation}\label{4.12}
S_{B_1}(\tau ) =\int_{\partial \M}\para{\partial_\nu u_1+iA_1(\nu)u_1} \overline{\ww}\ds,
\end{equation}
we get by applying  formula \eqref{2.5}
\begin{align}
S_{B_1}(\tau )=\int_\M\Delta_{A_1}u_1 \overline{\ww} \dv &- \int_\M u_1\overline{\Delta_{A_1}\ww} \dv \label{4.13}
\\
&+\int_{\p \M}\w\para{\overline{\p_\nu \ww+iA_1(\nu)\ww}} \ds.\nonumber
\end{align}
On the other hand, by a simple computation and using \eqref{4.2}, \eqref{4.3} and \eqref{4.4}, we get
\begin{align*}
\Delta_{A_1} \ww&=\Delta_{A_1}(e^{i\overline{\lambda}_\tau\psi}\bb_2)\cr
&=-\overline{\lambda}_\tau^2\ww+e^{i\overline{\lambda}_\tau\psi}\Delta_{A_1} (\bb_2)-2i\overline{\lambda}_\tau e^{i\overline{\lambda}_\tau\psi}\alpha_2\para{\seq{d\psi,d\beta_2}+i\seq{A_1,d\psi}\beta_2}\cr
&\qquad +2i\overline{\lambda}_\tau\beta_2e^{i\overline{\lambda}_\tau\psi}\para{\seq{d\psi,d\alpha_2}+\frac{\alpha_2}{2}\Delta\psi}\cr
&= -\overline{\lambda}_\tau^2\ww+e^{i\overline{\lambda}_\tau\psi}\Delta_{A_1} (\bb_2)-2i\overline{\lambda}_\tau e^{i\overline{\lambda}_\tau\psi}\alpha_2\para{-i\seq{A_2,d\psi}\beta_2+i\seq{A_1,d\psi}\beta_2}\cr
&=-\overline{\lambda}_\tau^2\ww+e^{i\overline{\lambda}_\tau\psi}\Delta_{A_1} (\bb_2)+2\overline{\lambda}_\tau e^{i\overline{\lambda}_\tau\psi}\bb_2\seq{A,d\psi}.
\end{align*}
Whence, in light of  \eqref{4.11}, we find
\begin{align*}
&\int_\M u_1\overline{\Delta_{A_1}\ww} \dv
\\
&\qquad =\int_\M \para{\w-R_{B_1}(\lambda_\tau^2)\para{e^{i\lambda_\tau\psi} k_1}}
\\
&\hskip 3cm \times \para{-\lambda_\tau^2\overline{\ww}+e^{-i\lambda_\tau\psi}\overline{\Delta_{A_1} (\bb_2)}+2\lambda_\tau e^{-i\lambda_\tau\psi}\overline{\bb_2}\seq{A,d\psi}}\dv,
\end{align*}
and, using again \eqref{4.11}, we get
\begin{align*}
\int_\M \Delta_{A_1}u_1 \overline{\ww} \dv &=-\int_\M\h_{B_1}u_1\overline{\ww} \dv+\int_\M q_1u_1\overline{\ww} \dv\cr
&=\int_\M \para{\w-R_{B_1}(\lambda_\tau^2) (e^{i\lambda_\tau\psi} k_1)}\para{-\lambda_\tau^2\overline{\ww}+q_1\overline{\ww}}\dv.
\end{align*}
We deduce that
\begin{align}
&\int_\M\Delta_{A_1}u_1 \overline{\ww} \dv - \int_\M u_1\overline{\Delta_{A_1}\ww} \dv \label{4.14}
\\
& \quad =\int_\M\para{\w-R_{B_1}(\lambda_\tau^2) (e^{i\lambda_\tau\psi} k_1)}\nonumber
\\
&\hskip 3cm \times \para{e^{-i\lambda_\tau\psi}\overline{\h_{B_1} (\bb_2)}-2\lambda_\tau e^{-i\lambda_\tau\psi}\overline{\bb_2}\seq{A,d\psi}}\dv\nonumber
\\
&\quad =\int_\M\bb_1\overline{\h_{B_1} (\bb_2)}\dv-2\lambda_\tau\int_\M\bb_1\overline{\bb_2}\seq{A,d\psi}\dv\nonumber
\\
&\quad -\int_\M R_{B_1}(\lambda_\tau^2) (e^{i\lambda_\tau\psi} k_1)\para{e^{-i\lambda_\tau\psi}\overline{\h_{B_1} (\bb_2)}-2\lambda_\tau e^{-i\lambda_\tau\psi}\overline{\bb_2}\seq{A,d\psi}}\dv.\nonumber
\end{align}
Moreover
\begin{align}
&\int_{\p \M}\w(\overline{\p_\nu \ww+iA_1(\nu)\ww}) \ds\label{4.15}
\\
&\hskip 3cm=\int_{\p\M}\bb_1\para{\p_\nu\overline{\bb_2}-iA_1(\nu)\overline{\bb_2}-i\lambda_\tau\overline{\bb_2}\p_\nu\psi}\ds.\nonumber
\end{align}
Finally,  we get \eqref{4.7} by combining \eqref{4.13}, \eqref{4.14} and \eqref{4.15}.
\medskip

The proof of \eqref{4.8} is quite similar to that of \eqref{4.7}. But, for the reader's convenience, we detail the proof of \eqref{4.8}. By a simple computation we find
\begin{align}
&\left(\mathcal{H}_{B_2} -\lambda_\tau ^2\right)\w=e^{i\lambda_\tau\psi}\h_{B_2}\para{\bb_1} \label{4.16}
\\
&\qquad +e^{i\lambda_\tau\psi}\bigg(\lambda_\tau^2 \bb_1 \para{\abs{d\psi}^2-1} -2i\lambda_\tau \beta_{A_1}\para{\seq{d\psi,d\alpha_1}+\frac{\alpha_1}{2}\Delta\psi } \nonumber \\
&\qquad -2i\lambda_{\tau} \alpha_1 \para{\seq{d\psi,d\beta_{A_1}} +i \seq{A_2,d\psi}\beta_{A_1}} \bigg) \nonumber
\end{align}
Taking into account  \eqref{4.2}-\eqref{4.3} and \eqref{4.4}, the right-hand side of \eqref{4.16} takes the form
\begin{equation}\label{4.17}
\left(\mathcal{H}_{B_2} -\lambda_\tau ^2\right)\w=e^{i\lambda_\tau\psi(x)}\para{\h_{B_2}\para{\bb_1}-2\lambda_\tau\seq{A,d\psi}\bb_1}
\equiv e^{i\lambda_\tau\psi(x)}k_2.
\end{equation}
Let $u_2$ be the solution of the BVP
\[
\left\{ 
\begin{array}{ll} 
\left(\mathcal{H}_{B_2}-\lambda_\tau ^2\right) u_2=0\quad  &\textrm{in}\; \M ,
\\ 
u_2=\w &\textrm{on}\; \p \M.
\end{array}
\right.
\]
As for $u_1$, we  split $u_2$ into two terms, $u_2=\w+v_2$, where $v_2$ is the solution of the BVP
\[
\left\{ 
\begin{array}{ll} 
\left(\mathcal{H}_{B_2}-\lambda_\tau ^2\right) v_2= -e^{i\lambda_\tau \psi} k_2\quad &\textrm{in}\; \M
\\ 
v_2=0 &\textrm{on}\; \p \M.
\end{array}
\right.
\]
Therefore
\begin{equation}\label{4.18}
u_2=\w-\left(H_{B_2}-\lambda_\tau ^2\right)^{-1}(e^{i\lambda_\tau\psi} k_2)=\w-R_{B_2}(\lambda_\tau^2) (e^{i\lambda_\tau\psi} k_2).
\end{equation}
Since
\begin{equation}\label{4.19}
S_{B_2}(\tau ) =\int_{\partial \M}\para{\partial_\nu u_2+iA_2(\nu)u_2} \overline{\ww}\ds, 
\end{equation}
we obtain, by applying  formula \eqref{2.5}, 
\begin{align}
S_{B_2}(\tau )&=\int_\M\Delta_{A_2}u_2 \overline{\ww} \dv - \int_\M u_2\overline{\Delta_{A_2}\ww} \dv \label{4.20}
\\
&\hskip 2cm+\int_{\p \M}\w(\overline{\p_\nu \ww+iA_2(\nu)\ww}) \ds.\nonumber
\end{align}
On the other hand, by using \eqref{4.2}, \eqref{4.3} and \eqref{4.4}, we find
\begin{equation}\label{4.21}
\Delta_{A_2} \ww=\Delta_{A_2}(e^{i\overline{\lambda}_\tau\psi}\bb_2)=-\overline{\lambda}_\tau^2\ww+e^{i\overline{\lambda}_\tau\psi}\Delta_{A_2} (\bb_2).
\end{equation}
Whence
\begin{align}
\int_\M u_2\overline{\Delta_{A_2}\ww} \dv=\int_\M &(\w-R_{B_2}(\lambda_\tau^2) (e^{i\lambda_\tau\psi} k_2)) \label{4.22}
\\
&\times \para{-\lambda_\tau^2\overline{\ww}+e^{-i\lambda_\tau\psi}\overline{\Delta_{A_2} (\bb_2)}}\dv \nonumber
\end{align}
and
\begin{align*}
\int_\M \Delta_{A_2}&u_2 \overline{\ww} \dv =-\int_\M\h_{B_2}u_2\overline{\ww} \dv+\int_\M q_2u_2\overline{\ww} \dv\cr
&=\int_M \para{\w-R_{B_2}(\lambda_\tau^2) (e^{i\lambda_\tau\psi} k_2)}\para{-\lambda_\tau^2\overline{\ww}+q_2\overline{\ww}}\dv.
\end{align*}
Thus,
\begin{align}
&\int_\M\Delta_{A_2}u_2 \overline{\ww} \dv - \int_\M u_2\overline{\Delta_{A_2}\ww} \dv \label{4.23}
\\
& =\int_\M\para{\w-R_{B_2}(\lambda_\tau^2) (e^{i\lambda_\tau\psi} k_2)}\para{e^{-i\lambda_\tau\psi}\overline{\h_{B_2} (\bb_2)}}\dv\nonumber
\\
&=\int_M\bb_1\overline{\h_{B_2} (\bb_2)}\dv-\int_\M R_{B_2}(\lambda_\tau^2) \para{e^{i\lambda_\tau\psi} k_2}\para{e^{-i\lambda_\tau\psi}\overline{\h_{B_1} (\bb_2)}}\dv\nonumber
\\
&= \int_M\bb_1\overline{\h_{B_2} (\bb_2)}\dv\nonumber
\\
& -\int_\M R_{B_2}(\lambda_\tau^2) \para{e^{i\lambda_\tau\psi} \para{\h_{B_2}\para{\bb_1}-2\lambda_\tau\seq{A,d\psi}\bb_1}}\para{e^{-i\lambda_\tau\psi}\overline{\h_{B_1} (\bb_2)}}\dv.\nonumber
\end{align}
Moreover, we have
\begin{align}
&\int_{\p \M}\w(\overline{\p_\nu \ww+iA_2(\nu)\ww}) \ds\label{4.24}
\\
&\hskip 2cm =\int_{\p\M}\bb_1\para{\p_\nu\overline{\bb_2}-iA_2(\nu)\overline{\bb_2}-i\lambda_\tau\overline{\bb_2}\p_\nu\psi}\ds.\nonumber
\end{align}
Inserting \eqref{4.24} and \eqref{4.23} in \eqref{4.20}, we obtain
\begin{align}
&S_{B_2}(\tau )=\int_{\p\M}\bb_1\para{\p_\nu\overline{\bb_2}-iA_2(\nu)\overline{\bb_2}-i\lambda_\tau\overline{\bb_2}\p_\nu\psi}\ds\label{4.25}
\\
&\hskip 4cm +\int_\M\bb_1\overline{\h_{B_2} (\bb_2)}\dv \nonumber
\\
& -\int_\M R_{B_2}(\lambda_\tau^2) \para{e^{i\lambda_\tau\psi} \para{\h_{B_2}\para{\bb_1}-2\lambda_\tau\seq{A,d\psi}\bb_1}}\para{e^{-i\lambda_\tau\psi}\overline{\h_{B_2} (\bb_2)}}\dv.\nonumber
\end{align}
This completes the proof of the Lemma.
\end{proof}
Subtracting side by side \eqref{4.7} and \eqref{4.8}, and using the fact that $A_1=A_2$ on $\p\M$, we obtain the following identity, that we will use later in the text.
\begin{multline}\label{4.26}
S_{B_1}(\tau)-S_{B_2}(\tau) = -2\lambda_\tau\int_\M\bb_1\overline{\bb_2}\seq{A,d\psi}\dv+\int_\M\bb_1(\overline{\h_{B_1}-\h_{B_2})(\bb_2)}\dv\cr
-\int_\M  R_{B_1}(\lambda_\tau^2) \para{e^{i\lambda_\tau\psi}\h_{B_1}\para{\bb_1}}\para{e^{-i\lambda_\tau\psi}\para{\overline{\h_{B_1} (\bb_2)}-2\lambda_\tau\overline{\bb_2}\seq{A,d\psi}}}\dv\cr
+\int_\M R_{B_2}(\lambda_\tau^2) \para{e^{i\lambda_\tau\psi} \para{\h_{B_2}\para{\bb_1}-2\lambda_\tau\seq{A,d\psi}\bb_1}}\para{e^{-i\lambda_\tau\psi}\overline{\h_{B_2} (\bb_2)}}\dv.
\end{multline}
\subsection{Solving the eikonal and transport equations}\ 

We  construct the phase function $\psi$ solution to the eikonal equation \eqref{4.2} and the amplitudes $\alpha_\ell$ and $\beta_\ell$, $\ell=1,2$, solutions to the transport equations \eqref{4.3}-\eqref{4.4}.

Let $y\in \p \M_1$. Denote points in $\M_1$ by $(r,\theta)$
where $(r,\theta)$ are polar normal coordinates in $\M_1$ with center
$y$. That is, $x=\exp_{y}(r\theta)$, where $r>0$ and
$$
\theta\in S_{y}^+\M_1=\set{\theta\in T_{y}\M_1,\,\,\abs{\theta}=1,\,\,\seq{\theta,\nu}<0}.
$$
In these coordinates (depending on the choice of $y$) the
metric has the form
$$
\widetilde{\g}(r,\theta)=\dd r^2+\g_0(r,\theta).
$$
If $u$ is a function in $\M$, set, for $r>0$ and $\theta\in S_y\M_1,$
$$
\widetilde{u}(r,\theta)=u(\exp_{y}(r\theta)),
$$
If $u$ is compactly supported, $\widetilde{u}$ is naturally extended by $0$ outside $\M$.

The geodesic distance  to $y$ provide an explicit solution of the eikonal equation \eqref{4.2}:
\begin{equation}\label{4.27}
\psi(x)=d_\g(x,y).
\end{equation}
Since $y\in \M_1\backslash\overline{\M}$, we have $\psi\in\mathcal{C}^\infty(\M)$ and
\begin{equation}\label{4.28}
\widetilde{\psi}(r,\theta)=r=d_\g(x,y).
\end{equation}

We now solve the transport equation \eqref{4.3}. To this and, recall that
if $f(r)$ is any function of the geodesic distance $r$, then
\begin{equation}\label{4.29}
\Delta_{\widetilde{\g}}f(r)=f''(r)+\frac{\varrho^{-1}}{2}\frac{\p
\varrho}{\p r}f'(r).
\end{equation}
Here $\varrho=\varrho(r,\theta)$ denotes the square of the volume element in geodesic polar coordinates.
In the new coordinates system, equation \eqref{4.3} takes the form
\begin{equation}\label{4.30}
\frac{\p \widetilde{\psi}}{\p
r}\frac{\p \widetilde{\alpha}}{\p
r}+\frac{1}{4}\widetilde{\alpha}\varrho^{-1}\frac{\p \varrho}{\p r}\frac{\p
\widetilde{\psi}}{\p r}=0.
\end{equation}
Thus $\widetilde{\alpha}$ satisfies
\begin{equation}\label{4.31}
\frac{\p \widetilde{\alpha}}{\p
r}+\frac{1}{4}\widetilde{\alpha}\varrho^{-1}\frac{\p \varrho}{\p r}=0.
\end{equation}
For  $\eta\in H^2(S_y^+\M)$, we seek $\widetilde{\alpha}$ in the form
\begin{equation}\label{4.32}
\widetilde{\alpha}(r,\theta)=\varrho^{-1/4}\eta(y,\theta).
\end{equation}
Direct computations yield
\begin{equation}\label{4.33}
\frac{\p \widetilde{\alpha}}{\p
r}(r,\theta)=-\frac{1}{4}\varrho^{-5/4}\frac{\p\varrho}{\p
r}\eta(y,\theta).
\end{equation}
Finally, \eqref{4.32} and \eqref{4.33} entail
\begin{equation}\label{4.34}
\frac{\p \widetilde{\alpha}}{\p
r}(r,\theta)=-\frac{1}{4}\varrho^{-1}\widetilde{\alpha}(r,\theta)\frac{\p\varrho}{\p
r}.
\end{equation}

In the rest of this subsection we are concerned with transport equation \eqref{4.4}. 
Using that, in polar coordinates, $\nabla\psi(x)$ can be expressed in term of $\dot{\gamma}_{y,\theta}(r)$ (see for instance \cite[Appendix C]{Be}), we have 
$$
\seq{\widetilde{A}_\ell(r,y,\theta),d\psi}=\seq{\widetilde{A}_\ell^\sharp(r,y,\theta),\nabla\psi}=\sigma_{A_\ell}(\varphi_r(y,\theta))=\widetilde{\sigma}_{A_\ell}(r,y,\theta).
$$
Consequently, in polar coordinates system, (\ref{4.4}) has the form
\begin{equation}\label{4.35}
\frac{\p \widetilde{\psi}}{\p
r}\frac{\p \widetilde{\beta}}{\p
r} +i\widetilde{\sigma}_{A_\ell}(r,y,\theta)\widetilde{\beta}=0,
\end{equation}
where $\widetilde{\sigma}_{A_\ell} (r,y,\theta):=\sigma_{A_\ell}(\Phi_r(y,\theta))=\seq{\dot{\gamma}_{y,\theta}(r),A_{\ell}^\sharp(\gamma_{y,\theta}(r))}$.
Thus $\widetilde{\beta}$ satisfies
\begin{equation}\label{4.36}
\frac{\p \widetilde{\beta}}{\p
r}+i\widetilde{\sigma}_{A_\ell}(r,y,\theta)\widetilde{\beta}=0.
\end{equation}
Thus, we can choose $\widetilde{\beta}$ defined as follows
$$
\widetilde{\beta}(y,r,\theta)=\exp\para{i\int_0^{\ell_+(y,\theta)}\widetilde{\sigma}_{A_\ell}(r+s,y,\theta)ds}.
$$
On the other words, we solved (\ref{3.4}).

In the remainder of this paper we  use the following notations: 
\begin{equation}\label{4.37}
\widetilde{\beta}_{A_\ell}(y,r,\theta)=\exp\para{i\int_0^{\ell_+(y,\theta)}\widetilde{\sigma}_{A_\ell}(r+s,y,\theta)ds}, \quad \ell=1,2,
\end{equation}
and
\begin{equation}\label{4.38}
\widetilde{\alpha}_1(r,\theta)=\varrho^{-1/4}\eta(y,\theta),\quad \widetilde{\alpha}_2(r,\theta)=\varrho^{-1/4}.
\end{equation}
\subsection{Asymptotic behavior of the boundary representation formula}
We discuss in this subsection the asymptotic behavior of $S_{B_1}(\tau)-S_{B_2}(\tau)$, as well as the asymptotic behavior of $[S_{B_1}(\tau)-S_{B_2}(\tau)]/\tau$, as $\tau\to\infty$.

As before, $B_\ell=(A_\ell,q_\ell)\in\mathscr{B}_r$, $\ell=1,2$ are so that  $A_\ell$ satisfy \eqref{t3e}. Set
$$
A(x)=(A_1-A_2)(x),\quad q(x)=(q_1-q_2)(x).
$$
Note that $A$, extended by $0$ outside $\M$, belongs to $\mathcal{C}^0(\M_{1},T^*\M_1)$. We also extend $q$ by $0$ outside $\M$. This extension, still denoted by $q$, is an element of $L^\infty (\M_1)$. 
\begin{lemma}\label{L.4.2}  For any $\eta\in H^2(S^+_y\M_1)$, we have 
\begin{equation}\label{4.39}
\lim_{\tau\to+\infty}\frac{S_{B_1}(\tau)-S_{B_2}(\tau)}{ \tau}=2i\int_{S_y^+\M_1}\para{e^{i\I_1A(y,\theta)}-1}\eta(y,\theta) \, \dd\omega_y(\theta).
\end{equation}
\end{lemma}
\begin{proof}
By the resolvent estimate, we have
\begin{equation}\label{4.40}
\norm{R_{B_\ell}(\lambda_\tau ^2)}_{\mathscr{L}(L^2(\M))}\leq \frac{1}{|\Im(\lambda_\tau^2)|}=\frac{1}{2\tau},\quad  \ell=1,2.
\end{equation}
Inequalities \eqref{4.40}  and \eqref{4.26} yield in a straightforward manner
\begin{align}
\lim_{\tau\to+\infty}\frac{S_{B_1}(\tau)-S_{B_2}(\tau)}{ \tau}&=2\int_\M\bb_1\overline{\bb_2}\seq{A,d\psi}\dv\label{4.41}
\\
&=2\int_\M\alpha_1\alpha_2\beta_{A_1}\overline{\beta_{A_2}}\seq{A,d\psi}\dv .\nonumber
\end{align}
Applying \eqref{eqi} and making the change variable $x=\exp_{y}(r\theta)$, with $r>0$ and
$\theta\in S_{y}\M_1$, we get
\begin{multline}\label{4.42}
2\int_\M \seq{A,d\psi} (\alpha_1\alpha_2)(x)(\beta_{A_1}\overline{\beta}_{A_2})(x)  \dv \cr
=2\int_{S^+_{y}\M_1}\!\int_0^{\ell_+(y,\theta)}\widetilde{\sigma}_A(r,y,\theta)(\widetilde{\alpha}_1\widetilde{\alpha}_2)(r,\theta)(\widetilde{\beta}_{A_1}\overline{\widetilde{\beta}}_{A_2})(r,\theta)
\varrho^{1/2} \, \dd r \, \dd\omega_y(\theta) \cr
=2\int_{S^+_{y}\M_1}\!\int_0^{\ell_+(y,\theta)} \widetilde{\sigma}_A(r,y,\theta)\widetilde{\beta}_{A_1}(r,\theta)\overline{\widetilde{\beta}}_{A_2}(r,\theta)\eta(y,\theta) \dd r \,
\dd\omega_y(\theta)\cr
=\int_{S_y^+\M_1}\int_0^{\ell_+(y,\theta)}  \widetilde{\sigma}_A(r,y,\theta) \exp\para{i\int_0^{\ell_+(y,\theta)} \widetilde{\sigma}_A(r+s,y,\theta) ds}\eta(y,\theta)
\dd r \,
\dd\omega_y(\theta).
\end{multline}
Also
\begin{multline}\label{4.43}
\int_0^{\ell_+(y,\theta)} \widetilde{\sigma}_A(r,y,\theta) \exp\para{i\int_0^{\ell_+(y,\theta)} \widetilde{\sigma}_A(r+s,y,\theta) ds}
\dd r \cr
=-i\int_0^{\ell_+(y,\theta)} \p_r\cro{\exp\para{i\int_0^{\ell_+(y,\theta)} \widetilde{\sigma}_A(r+s,y,\theta) ds}}\dd r\cr
=i\cro{\exp\para{i\int_0^{\ell_+(y,\theta)} \widetilde{\sigma}_A(s,y,\theta) ds}-1},
\end{multline}
entailing
\begin{align*}
2\int_\M \seq{A,d\psi} (\alpha_1\alpha_2)(x)&(\beta_{A_1}\overline{\beta}_{A_2})(x)  \dv
\\
&=2i\int_{S_y^+\M_1}\para{\exp\para{i\I_1A(y,\theta)}-1}\eta(y,\theta)\dd\omega_y(\theta).
\end{align*}
This in \eqref{4.41} gives the expected inequality. 
\end{proof}
\begin{lemma}\label{L.4.3} 
Assume  that $A_1=A_2$. Then, for any $\eta\in H^2(S^+_y\M_1)$, we have
\begin{equation}\label{4.45}
\lim_{\tau\to+\infty}\para{S_{B_1}(\tau)-S_{B_2}(\tau)}=\int_{S_y^+\M_1}\I_0(q)(y,\theta)\eta(y,\theta) \dd\omega_y(\theta).
\end{equation}
\end{lemma}
\begin{proof} 
Since $A_1=A_2$,  \eqref{4.26} is reduced to the following formula
\begin{align}
S_{B_1}(\tau)&-S_{B_2}(\tau) =\int_\M q(x)\bb_1(x)\overline{\bb_2}(x)\dv \label{4.46}
\\
&-\int_\M R_{B_1}(\lambda_\tau^2) \para{e^{i\lambda_\tau\psi}\h_{B_1}\para{\bb_1}}\para{e^{-i\lambda_\tau\psi}\overline{\h_{B_1} (\bb_2)}}\dv\nonumber
\\
&\qquad +\int_\M R_{B_2}(\lambda_\tau^2) \para{e^{i\lambda_\tau\psi} \h_{B_2}(\bb_1)}\para{e^{-i\lambda_\tau\psi}\overline{\h_{B_2} (\bb_2)}}\dv.\nonumber
\end{align}
Once again the resolvent estimate enables us to get
\begin{equation}\label{4.47}
\lim_{\tau\to+\infty}\para{S_{B_1}(\tau)-S_{B_2}(\tau)}=\int_{\M} q(x)(\alpha_1\alpha_2)(x)\dv.
\end{equation}
We complete the proof by mimicking the end of the previous proof in order to obtain
\begin{equation}\label{4.48}
\int_{\M} q(x)(\alpha_1\alpha_2)(x)\dv=\int_{S_y^+\M_1}\I_0(q)(y,\theta)\eta(y,\theta)\dd\omega_y(\theta).
\end{equation}
This completes the proof.
\end{proof}
\section{Proof of the main results}
\setcounter{equation}{0}

\subsection{Asymptotic behavior of the spectral data}
Prior to the completion of the proof of Theorems \ref{t1} and \ref{t2}, we establish some technicals lemmas. Assumptions and notations are the same as in the preceding one.

\begin{lemma}\label{L.5.1} 
For $t\in [0,1/2)$ and $\ell=1,2$, we have 
\begin{equation}\label{5.1}
\sum_{k\ge 1} k^{2t/n}\Abs{\frac{\seq{\varphi_{1,\tau}^*,\psi_{\ell,k}}}{\lambda_{\ell,k}-\lambda_\tau^2 }}^2\le C_\ell\tau^{2t}\norm{\eta}^2_{H^2(S_y^+\M_1)}
\end{equation}
and 
\begin{equation}\label{5.2}
\sum_{k\ge 1}k^{2t/n} \Abs{\frac{\seq{\ww,\psi_{2,k}}}{\lambda_{\ell,k}-\lambda_\tau^2 }}^2\leq C_\ell \tau^{2t},  
\end{equation}
the constant $C_\ell$ depends on $t$, $\M$, $r$ and $B_\ell$ if $t>0$,  and  it is independent on $B_\ell$ when $t=0$.
\end{lemma}
\begin{proof}  By Lemma \ref{L.3.1} the solution of the boundary value problem \eqref{3.2}, with $f=\w$, $\lambda =\lambda_\tau$ and $B=B_1$, is given by the series
\begin{equation}\label{5.3}
u_{1}(\lambda_\tau )=\sum_{k\ge 1} \frac{\seq{\w,\psi_{1,k}}}{\lambda_\tau^2-\lambda_{1,k}}\phi_{1,k}.
\end{equation}
If $\mu =2r+1$, then the operator $H_{B_\ell}+\mu$ is positive. Indeed, for $u\in H^1_0(\M)$, we have
\begin{align*}
\int_\M \para{H_{B_\ell}+\mu}u\overline{u}\dv &=\int_\M\abs{\nabla_{A_\ell} u}^2\dv+\int_\M(q_\ell+\mu)\abs{u}^2\dv\cr
&\geq \int_\M\abs{d u}^2\dv+(\mu-\norm{q_\ell}_\infty-2\norm{A_\ell}_\infty)\int_\M\abs{u}^2\dv.
\end{align*}
Since $\mathscr{D}((H_{B_\ell}+\mu)^{\frac{1}{2}})=H^1_0(\M)$ we have, by interpolation, $\mathscr{D}((H_{B_\ell}+\mu )^{\frac{t}{2}})=H^{t}_0(\M)=H^{t}(\M)$ (e.g. \cite[Chapter 1, Theorems 11.1 and 11.6]{LM1}). Whence, for $w\in H^{t}(\M)$, we have
\begin{equation}\label{5.4} 
\sum_{k\ge 1}(1+|\lambda_{\ell,k}|)^{t}|(w,\phi_{\ell,k})|^2\le C_\ell\norm{w}_{H^{t}(\M)}^2,\quad \ell=1,2,
\end{equation}
the constant $C_\ell$  only depends on $t$, $r$ and $\M$ and $B_\ell$. 

On the other hand, we get from \eqref{4.11}
\begin{align}
\norm{u_{1}(\lambda_\tau )}_{H^t(\M)}&\le \norm{\w}_{H^t(\M)}+\norm{R_{B_1}(\lambda_\tau^2)(e^{i\lambda_\tau\psi} \mathcal{H}_{B_1}\bb_1)}_{H^{t}(\M)}\label{5.5}
\\
&\leq C\tau^{t}\norm{\eta}_{H^2(S_y^+\M_1)}.\nonumber
\end{align}
Here again the constant $C$ only depends on $t$, $r$, $\M$ and $B_1$, where we used that $\exp_y^{-1}(M)\subset \{r\theta:\ r>0,\ \theta\in S_y^+(M_1)\}$ in order to restrict the norm of $\eta$ to $S_y^+\M_1$.\\
This estimate and \eqref{5.4} with $w=u_{1}(\lambda_\tau )$ and $\ell=1$ entail
\begin{equation}\label{5.6}
\sum_{k\ge 1}(1+|\lambda_{1,k}|)^{t}|(u_{1}(\lambda_\tau ),\phi_{1,k})|^2\le C_1\tau^{2t}\norm{\eta}^2_{H^2(S_y^+\M_1)}.
\end{equation}
We get the first estimate \eqref{5.1} for $\ell=1$, by using  \eqref{A1} in Appendix \ref{appendix} and the identity
\begin{equation}\label{5.7}
(u_{1}(\lambda_\tau ),\phi_{1,k})=\frac{\seq{\w,\psi_{1,k}}}{\lambda_\tau^2-\lambda_{1,k}}.
\end{equation}
To prove the first inequality \eqref{5.1} for $\ell=2$, we consider $u_{2}(\lambda_\tau )$, the solution of the BVP \eqref{3.2} when $\lambda=\lambda_\tau$, $f=\w$ and $B=B_2$. By Lemma \ref{L.3.1}, this solution is given by the series
\begin{equation}\label{5.8}
u_{2}(\lambda_\tau )=\sum_{k\ge 1} \frac{\seq{\w,\psi_{2,k}}}{\lambda_\tau^2-\lambda_{2,k}}\phi_{2,k}.
\end{equation}
On the other hand, we get from \eqref{4.18} and \eqref{3.1}
\begin{align}
\norm{u_{2}(\lambda_\tau )}_{H^t(\M)} &\le  \norm{\w}_{H^t(\M)}\label{5.9}
\\
&+\norm{R_{B_2}(\lambda_\tau^2)(e^{i\lambda_\tau\psi} \para{\mathcal{H}_{B_2}(\bb_1)-2\lambda_\tau\seq{A,d\psi}\bb_2}}_{H^{t}(\M)}\nonumber
\\
&\leq  C\para{\tau^{t}+\frac{\abs{\lambda_\tau}}{\tau^{1-t}}}\norm{\eta}_{H^2(S_y^+\M_1)}\leq  C\tau^t \norm{\eta}_{H^2(S_y^+\M_1)}.\nonumber
\end{align}
Applying again  \eqref{5.4} with $w=u_{2}(\lambda_\tau )$ and $\ell=2$ entail
\begin{equation}\label{5.10}
\sum_{k\ge 1}(1+|\lambda_{2,k}|)^{t}|(u_{2}(\lambda_\tau ),\phi_{2,k})|^2\le C_2\tau^{2t}\norm{\eta}^2_{H^2(S_y^+\M_1)}.
\end{equation}
Since
\begin{equation}\label{5.11}
(u_{2}(\lambda_\tau ),\phi_{2,k})=\frac{\seq{\w,\psi_{2,k}}}{\lambda_\tau^2-\lambda_{2,k}}.
\end{equation}
we obtain \eqref{5.1} with $\ell=2$.\\
The second inequality of \eqref{5.2} is proved similarly.
\end{proof}
Let us recall some notations that we introduced in Section 3. For $f\in H^{3/2}(\p \M)$ fixed  and  $\lambda, \mu\in  \rho (H_{B_1})\cap \rho (H_{B_2})$, if $u_{\ell}(\lambda )$ (resp. $u_{\ell}(\mu )$) is the solution of the boundary value problem \eqref{3.2}  for $B=B_\ell$ (resp. $B=B_\ell$ and $\lambda=\mu$), $\ell=1,2$, we have posed
\begin{align}\label{5.12}
&w_{\ell}(\lambda ,\mu )=u_{\ell}(\lambda )-u_{\ell}(\mu ),\cr
& w_{1,2}(\mu )=u_{1}(\mu )-u_{2}(\mu ).
\end{align}
Let
\begin{equation}\label{5.13}
\mathcal{K}(\tau ,\mu ,f)=\para{\p_\nu+iA_1(\nu)}w_{1}(\lambda_\tau ,\mu)-\para{\p_\nu+iA_2(\nu)}w_{2}(\lambda_\tau ,\mu)\quad \textrm{on}\;\p \M.
\end{equation}
Then, by \eqref{3.18}, we obtain
\begin{equation}\label{5.14}
\mathcal{K}(\tau,\mu,f) =\sum_{k \ge 1} \left[\frac{(\mu - \lambda_\tau ^2)\seq{f,\psi_{1,k}}}{(\lambda_\tau ^2 - \lambda_{1,k}) (\mu - \lambda_{1,k})} \psi_{1,k}-\frac{(\tau -\lambda_\tau^2)\seq{f,\psi_{2,k}}}{(\lambda_\tau^2 - \lambda_{2,k})(\mu - \lambda_{2,k})} \psi_{2,k}\right].
\end{equation}
We define
\begin{equation}\label{5.15}
\mathcal{L}(\tau,\mu )=\seq{\mathcal{K}(\tau,\mu, \w), \ww}.
\end{equation}
From \eqref{5.14}, we get
\begin{equation}\label{5.16}
\mathcal{L}(\tau,\mu )=\sum_{k \ge 1}(\mu - \lambda_\tau ^2)\left[\frac{\seq{\w,\psi_{1,k}}\seq{\psi_{1,k},\ww}}{ (\lambda_\tau^2 - \lambda_{1,k})(\mu - \lambda_{1,k})} \, -\frac{\seq{\w,\psi_{2,k}}\seq{\psi_{2,k},\ww}}{ (\lambda_\tau^2 - \lambda_{2,k})(\mu - \lambda_{2,k})} \right]. 
\end{equation}

Define
\begin{equation}\label{5.18}
\mathcal{L}^*(\tau )=\sum_{k\ge 1}\mathcal{L}^*_{1,k}(\tau)+\sum_{k\ge 1}\mathcal{L}^*_{2,k}(\tau)+\sum_{k\ge1}\mathcal{L}^*_{3,k}(\tau),
\end{equation}
with
$$
\mathcal{L}^*_{1,k}(\tau)=\frac{\left\langle \w,\psi_{1,k}-\psi_{2,k} \right\rangle \seq{ \psi_{1,k},\ww}}{ \lambda_\tau^2 - \lambda_{1,k}}
$$
$$
\mathcal{L}^*_{2,k}(\tau)=\frac{\left\langle \w,\psi_{2,k} \right\rangle \seq{ \psi_{1,k}-\psi_{2,k},\ww}}{ \lambda_\tau^2 - \lambda_{1,k}},
$$
$$
\mathcal{L}^*_{3,k}(\tau)=\left\langle \w,\psi_{2,k} \right\rangle \seq{\psi_{2,k},\ww}\left(\frac{1}{(\lambda_\tau^2 - \lambda_{1,k})} -\frac{1}{ (\lambda_\tau^2 - \lambda_{2,k})}\right).
$$
\begin{lemma}\label{L.5.2} Under assumption \eqref{1.16}, $\mathcal{L}(\tau,\mu )$ converge to $\mathcal{L}^*(\tau )$ as $\mu\to-\infty$ and, for $t\in[0,1/2)$,  we have
\begin{equation}\label{5.17} 
\limsup_{\tau\to \infty} \tau^{-t}|\mathcal{L}^*(\tau )|\leq C\norm{\eta}_{H^2(S_y^+\M_1)}\limsup_{k\to \infty}k^{-t/n}|\lambda_{1,k}-\lambda_{2,k}| .
\end{equation}
 \end{lemma}
\begin{proof} 
We split $\mathcal{L}(\tau,\mu )$ into three series
$$
\mathcal{L}(\tau,\mu)=\sum_{k\ge1} \mathcal{L}_{1,k}(\mu,\tau)+\sum_{k\ge1} \mathcal{L}_{2,k}(\mu,\tau)+\sum_{k\ge1} \mathcal{L}_{3,k}(\mu,\tau),
$$
with
$$
\mathcal{L}_{1,k}(\tau ,\mu )=(\mu - \lambda_\tau ^2)\frac{\seq{\w,\psi_{1,k}-\psi_{2,k}} \seq{\psi_{1,k},\ww}}{ (\lambda_\tau^2 - \lambda_{1,k})(\mu - \lambda_{1,k})},
$$
$$
\mathcal{L}_{2,k}(\tau ,\mu )=(\mu - \lambda_\tau^2)\frac{\seq{ \w,\psi_{2,k}}\seq{ \psi_{1,k}-\psi_{2,k},\ww}}{ (\lambda_\tau ^2 - \lambda_{1,k})(\mu - \lambda_{1,k})},
$$
\begin{align*}
\mathcal{L}_{3,k}(\tau ,\mu )&=(\mu - \lambda_\tau^2)\seq{ \w,\psi_{2,k} } \seq{\psi_{2,k},\ww}
\\
&\hskip 1.5cm\times \left(\frac{1}{ (\lambda_\tau^2 - \lambda_{1,k})(\mu - \lambda_{1,k})} -\frac{1}{ (\lambda_\tau^2 - \lambda_{2,k})(\mu - \lambda_{2,k})}\right).
\end{align*}
Under assumption \eqref{1.16} and in light of \eqref{5.1}, we can see that the series in $\mathcal{L}_{1,k}(\tau ,\mu )$, $\mathcal{L}_{2,k}(\tau ,\mu )$ and $\mathcal{L}_{3,k}(\tau ,\mu )$ converge uniformly with respect to $\mu \ll -1$.  Therefore, $\mathcal{L}(\tau,\mu )$ converge to $\mathcal{L}^*(\tau )$ as $\mu\to-\infty$. 

We have
\begin{equation}\label{5.19}
|\mathcal{L}^*_{1,k}(\tau)|\leq \norm{ \w}_{L^2(\p \M)}\norm{\psi_{1,k}-\psi_{2,k}}_{L^2(\p \M)} \Abs{\frac{ \seq{\psi_{1,k},\ww}}{\lambda_\tau^2 - \lambda_{1,k}}},
\end{equation}
\begin{multline}\label{5.20}
|\mathcal{L}^*_{2,k}(\tau)|\le  \frac{\norm{ \w}_{L^2(\p \M)}\norm{ \ww}_{L^2(\p \M)}\norm{\psi_{1,k}-\psi_{2,k}}_{L^2(\p \M)}^2}{ |\lambda_\tau^2 - \lambda_{1,k}|}\cr
 +\norm{ \ww}_{L^2(\p \M)}\norm{\psi_{1,k}-\psi_{2,k}}_{L^2(\p \M)} \Abs{\frac{ \seq{ \w,\psi_{1,k}}}{ \lambda_\tau^2 - \lambda_{1,k}}},
\end{multline}
\begin{multline}\label{5.21}
|\mathcal{L}^*_{3,k}(\tau)|\le \norm{ \w}_{L^2(\p \M)}\norm{\psi_{1,k}-\psi_{2,k}}_{L^2(\p \M)}\frac{|\lambda_{2,k}-\lambda_{1,k}|}{|\lambda_\tau^2 - \lambda_{2,k}|}  \Abs{\frac{ \seq{ \psi_{2,k},\ww}}{ \lambda_\tau^2 - \lambda_{2,k}}}\cr
+|\lambda_{2,k}-\lambda_{1,k}|\Abs{\frac{\seq{ \w,\psi_{1,k}}}{ \lambda_\tau^2 - \lambda_{1,k}}}\Abs{\frac{\seq{\psi_{2,k},\ww}}{ \lambda_\tau^2 - \lambda_{2,k}}}.
\end{multline}
But
\begin{equation}\label{5.22} 
\sup_{\tau>1}\norm{\w}_{L^2(\p \M)}\leq \norm{\bb_1}_{L^2(\p \M)}\leq C\norm{\eta}_{H^2(S_y\M_1)}
\end{equation}
and
\begin{equation}\label{5.22.1} 
\sup_{\tau>1}\norm{\ww}_{L^2(\p \M)}\leq \norm{\bb_1}_{L^2(\p \M)}\leq C,
\end{equation}
the constant $C$ only depends on $\M$. This estimate entails in particular that
$$
\limsup_{\tau\to+\infty}\tau^{-t}|\mathcal{L}^*_{1,k}(\tau)|=0,\quad k\geq1.
$$
Thus, for an arbitrary positive integer $n_1$, we get
$$
\limsup_{\tau\to+\infty}\tau^{-t}\sum_{k=1}^\infty |\mathcal{L}^*_{1,k}(\tau)|=\limsup_{\tau\to+\infty}\tau^{-t}\sum_{k=n_1}^\infty |\mathcal{L}^*_{1,k}(\tau)|.
$$
This  estimate together with \eqref{5.1}, \eqref{5.19}, \eqref{5.22} and \eqref{5.22.1} imply
\begin{align*}
\tau^{-t}\sum_{k=n_1}^\infty |\mathcal{L}^*_{1,k}(\tau)|&\le  C\left(\sup_{\tau>1}\tau^{-2t}\sum_{k=1}^\infty k^{2t/n}\Abs{\frac{ \seq{\psi_{1,k},\ww}}{ \lambda_\tau^2 - \lambda_{1,k}}}^2\right)^{1/2} 
\\
&\hskip 2cm \times\left(\sum_{k=n_1}^\infty k^{-2t/n}\norm{\psi_{1,k}-\psi_{2,k}}_{L^2(\p \M)}^2\right)^{1/2}
\\
&\le C\left(\sum_{k=n_1}^\infty k^{-2t/n}\norm{\psi_{1,k}-\psi_{2,k}}_{L^2(\p \M)}^2\right)^{1/2},
\end{align*}
the constant $C$ is independent on $\tau$.
Since the last term goes to zero as $n_1$ tends to $\infty$ by \eqref{1.18}, we easily get
\begin{equation}\label{5.23}
\limsup_{\tau\to+\infty}\tau^{-t}\sum_{k=1}^\infty |\mathcal{L}^*_{1,k}(\tau)|=0.
\end{equation}
In the sequel, we use the following useful observation:  for $r>1$ the map $\tau\mapsto |\lambda_\tau ^2 -r|$ reach its minimum  at $\tau=\sqrt{r-1}$. Hence
$$
|\lambda_\tau ^2 -r|\geq 2\sqrt{r-1},\quad \tau>0.
$$
This observation together with \eqref{5.1}, \eqref{5.20} and  \eqref{A1} in Appendix \ref{appendix} yield
\begin{multline*}
\underset{\tau\to+\infty}{\limsup}\,\tau^{-t}\sum_{k=1}^\infty |\mathcal{L}^*_{2,k}(\tau)|= \underset{\tau\to+\infty}{\limsup}\,\tau^{-t}\sum_{k=n_1}^\infty |\mathcal{L}^*_{2,k}(\tau)|
 \cr
 \leq C\sum_{k=n_1}^\infty k^{-1/n}\norm{\psi_{1,k}-\psi_{2,k}}_{L^2(\p \M)}^2
\cr
+C\left(\sup_{\tau>1}\tau^{-2t}\sum_{k=1}^\infty k^{2t/n}\Abs{\frac{ \seq{\psi_{1,k},\w}}{ \lambda_\tau^2 - \lambda_{1,k}}}^2\right)^{1/2} \left(\sum_{k=n_1}^\infty k^{-2/n}\norm{\psi_{1,k}-\psi_{2,k}}_{L^2(\p \M)}^2\right)^{1/2}\cr
\leq C\sum_{k=n_1}^\infty k^{-2t/n}\norm{\psi_{1,k}-\psi_{2,k}}_{L^2(\p \M)}^2+C\left(\sum_{k=n_1}^\infty k^{-2t/n}\norm{\psi_{1,k}-\psi_{2,k}}_{L^2(\p \M)}^2\right)^{1/2}.
\end{multline*}
Then, using again the fact that $n_1$ is arbitrary and \eqref{1.16}, we find
\begin{equation}\label{5.24}
\limsup_{\tau\to+\infty}\tau^{-t}\sum_{k=1}^\infty |\mathcal{L}^*_{2,k}(\tau)|=0.
\end{equation}
The same argument as before enables us to obtain
\begin{equation}\label{5.25}
\limsup_{\tau\to+\infty}\tau^{-t}\sum_{k=1}^\infty |\mathcal{L}^*_{3,k}(\tau)|\leq C\norm{\eta}_{H^2(S_y^+\M_1)}\limsup_{k\to+\infty}k^{-t/n}|\lambda_{1,k}-\lambda_{2,k}|.
\end{equation}
The expected result follows from \eqref{5.23}, \eqref{5.24} and \eqref{5.25}.
\end{proof}

\subsection{End of the proof of the main results}
We are now ready to complete the proof of Theorems \ref{t1} and \ref{t2}.
\begin{proof}[Proof of Theorem \ref{t1}]
Since  $A_\ell$, $\ell=1,2$, satisfy \eqref{t3e} and $w_{1,2}(\mu)=0$ on $\p\M$, we easily obtain the following identity, useful in the sequel,
\begin{equation}\label{5.26}
\mathcal{K}(\tau,\mu,\w)=(\partial_\nu +iA_1(\nu)) u_{1}(\lambda )-(\partial_\nu +iA_2(\nu)) u_{2}(\lambda )- \partial_\nu w_{1,2}(\mu )\quad \textrm{on}\,\, \p \M.
\end{equation}
By formula \eqref{5.15} we get
\begin{multline}\label{5.27}
\mathcal{L}(\tau,\mu) =\int_{\p\M} \mathcal{K}(\tau,\mu,\w)\overline{\ww}\ds
\cr
=\int_{\p\M} \para{\partial_\nu +iA_1(\nu)} u_{1}(\lambda )\overline{\ww}\ds-\int_{\p\M} \para{\partial_\nu +iA_2(\nu)} u_{2}(\lambda )\overline{\ww}\ds 
\cr
\qquad-\int_{\p\M}\partial_\nu w_{1,2}(\mu )\overline{\ww} \ds
\cr
=\int_{\p\M} \Lambda_{B_1}(\lambda_\tau^2)\w\overline{\ww}\ds-\int_{\p\M} \Lambda_{B_2}(\lambda_\tau^2)\w\overline{\ww}\ds\hskip 2.2cm
\cr 
-\int_{\p\M}\partial_\nu w_{1,2}(\mu )\overline{\ww} \ds\cr
=S_{B_1}(\tau)-S_{B_2}(\tau)-\int_{\p\M}\partial_\nu w_{1,2}(\mu )\overline{\ww} \ds \hskip 4.4cm.
\end{multline}
 According to Lemmas \ref{L.3.2} and \ref{L.5.2}, formula \eqref{5.27} and passing to the limit as $\mu$ goes to $-\infty$, we get
\begin{equation}\label{5.28}
S_{B_1}(\tau)-S_{B_2}(\tau)=\mathcal{L}^*(\tau).
\end{equation}
Furthermore, from \eqref{5.17} we have $\tau^{-t}\left(S_{B_1}(\tau)-S_{B_2}(\tau)\right)$ is bounded for $\tau >1$ and $t\in[0,1/2)$. Then $\tau^{-1}\left(S_{B_1}(\tau)-S_{B_2}(\tau)\right)$ goes to zero as $\tau$ tends to $\infty$. This in \eqref{4.39} yields, 
\begin{equation}\label{5.29}
\int_{S_y^+\M_1}\para{e^{i\I_1A(y,\theta)}-1}\eta(y,\theta) \, \dd\omega_y(\theta)=0.
\end{equation}
Since $\eta$ is arbitrary in $H^2(S_y\M)$,  we obtain that $\I_1A(y,\theta)\in 2\pi\mathbb Z$ for any $\theta \in S_y^+\M_1$. 
On the other hand, since $\p M_1$ is strictly convex, $S^+_y\M_1 \ni \theta \mapsto \ell_+(y,\theta)$ is continuous, and letting $\theta$ tend to a tangent direction $\theta_0 \in S_y \p M_1$ we get
\[ \lim_{\theta \to \theta_0} \ell_+(y,\theta) = 0\]
hence
\[ 2\pi m = \lim_{\theta \to \theta_0} \I_1A(y,\theta) = 0 \]
and therefore

\begin{equation}\label{5.30}
\mathcal{I}_1A(y,\theta)=0,\quad \theta\in S_y^+\M_1
\end{equation}
which implies that $\I_1A=0$, because $y\in \p \M_1$ is arbitrary. From \eqref{2.15}, we deduce that the solenoidal part $A^s$ in the Hodge decomposition of the $1$-form $A$ is equal to zero. This completes the proof of the first part of Theorem \ref{t1}. 

Now let us consider the second part of the theorem. For this purpose, we assume that condition \eqref{1.18} is fulfilled and we would like to show that $q_1=q_2$. Note first that the condition $A^s=0$ implies $dA=0$ and, since $\M_1$ is simply connected, there exists $\varphi\in W^{3,\infty}(\M_1)$ such that $d\varphi =A$. Since $A=0$ on  $M_1\setminus M $ by eventually extracting a constant to $\varphi$ we may assume that $\varphi=0$ on $M_1\setminus M $. In particular we have $\varphi_{|\partial\M}=\partial_\nu\varphi_{|\partial\M}=0$. Let $B_3=(A_1,q_2)$. Applying \eqref{1.5}, we deduce that
$$e^{-i\varphi}H_{B_2}e^{i\varphi}=H_{B_3}.$$
In particular, for $\lambda_{3,k}$, $k\geq1$, the non-decreasing sequence of eigenvalues of $H_{B_3}$ we have $\lambda_{3,k}=\lambda_{2,k}$ and $\phi_{3,k}=e^{-i\varphi}\phi_{2,k}$ corresponds to an orthonormal basis of eigenfunctions of $H_{B_3}$. Moreover, fixing
$\psi_{3,k}=\para{\p_\nu+iA_2(\nu)}\phi_{3,k}$, we deduce that
$$\begin{aligned}\psi_{3,k}(x)=\para{\p_\nu+iA_1(\nu)}e^{-i\varphi}\phi_{2,k}(x)&=e^{-i\varphi}\para{\p_\nu+iA_1(\nu)-i\partial_\nu\varphi}\phi_{2,k}(x)\\
\ &=\para{\p_\nu+iA_2(\nu)}\phi_{2,k}(x)=\psi_{2,k}(x),\ x\in\partial \M.\end{aligned}$$
Combining this with \eqref{1.18}, we deduce that
 $$\lim_{k\to+\infty}\abs{\lambda_{1,k}-\lambda_{3,k}}=0,\quad  \textrm{and}\;\;  \sum_{k\ge 1}\norm{\psi_{1,k}-\psi_{3,k}}_{L^2(\p \M)}^2<\infty.$$
In view of this gauge invariance property, from now on, without lost of generality, we may assume that $A_1=A_2$. According to \eqref{1.18},  with $t=0$, the right hand side of \eqref{5.17} is equal to zero. 
\end{proof}

\begin{proof}[Proof of Theorem \ref{t2}]  We already proved that $dA_1=dA_2$ in Theorem \ref{t1} and according to the gauge invariance property of the boundary spectral data, without lost of generality, we may assume that $A_1=A_2$. Then a straightforward application of the min-max principle yields
\begin{equation}\label{5.32}
|\lambda_{1,k}-\lambda_{2,k}|\leq \norm{q_1-q_2}_{L^\infty(\M)}.
\end{equation}
In that case \eqref{1.19} is reduced to
\begin{equation}\label{5.33}
\sum_{k\geq1}\norm{\psi_{1,k}-\psi_{2,k}}^2_{L^2(\p \M)}<\infty.
\end{equation}
Combining this with \eqref{4.45}, \eqref{5.17} for $t=0$ (which is valid in the present case) and taking into account that
\begin{equation}\label{5.34}
\limsup_{\tau\to+\infty}\left|S_{B_1}(\tau)-S_{B_2}(\tau)\right|=\limsup_{\tau\to+\infty} |\mathcal{L}^*(\tau )|,
\end{equation}
we obtain, for any $\eta\in H^2(S^+_y\M_1)$ real valued, that
\begin{equation}\label{5.35}
\bigg|\int_{S_y^+\M_1}\I_0(q)(y,\theta)\eta (y,\theta)\dss\bigg| \leq C\norm{\eta}_{H^2( S_y^+\M_1)}\limsup_{k\to+\infty}|\lambda_{1,k}-\lambda_{B_2^k}|.
\end{equation}
Since $q\in H^1(\M_1)$, by the smoothing effect of the normal operator $N_0=\I_0^\ast\I_0$ (see \eqref{2.22}), $N_0q\in H^2(M)$ and
\begin{equation}\label{5.36}
\norm{N_0(q)}_{H^2(\M_1)}\leq C\norm{q}_{H^1(\M)}\leq Cr'.
\end{equation}
Since $\I_0:\ H^2(\M_1)\rightarrow H^2(\p_+S\M_1)$ is bounded, we can take  $\eta=\I_0N_0(q)$. We  integrate with respect to $y\in\p \M_1$ the left hand side \eqref{5.35} in order to get
$$
\int_{\p_+S\M_1}\I_0(q)(y,\theta)\eta(y,\theta)\dss=\int_{\M_1}|N_0(q)|^2\dv=\norm{N_0(q)}_{L^2(\M_1)}^2.
$$
Combined with \eqref{5.35}, this inequality entails
\begin{equation}\label{5.37}
\norm{N_0(q)}_{L^2(\M_1)}^2\leq C\norm{\I_0N_0(q)}_{H^2(\p_+ S\M_1)}\limsup_{k\to+\infty}|\lambda_{1,k}-\lambda_{2,k}|.
\end{equation}
On the other hand,  it follows from  \eqref{5.36} 
\begin{equation}\label{5.38}
\norm{\I_0N_0(q)}_{H^2(\p_+ S\M_1)}\leq C\norm{N_0(q)}_{H^2(\M_1)}\leq C',
\end{equation}
the constants $C$ and $C'$ only depend on $\M$ and $r$.  This \eqref{5.37} and \eqref{5.38}, give
\begin{equation}\label{5.39}
\norm{N_0(q)}_{L^2(\M_1)}^2\leq C\limsup_{k\to+\infty}|\lambda_{1,k}-\lambda_{2,k}|,
\end{equation}
the constant $C$ only depends on $\M$ and $r'$. We complete the proof by using the interpolation inequality
$$\norm{N_0(q)}_{H^1(\M_1)}\leq C\norm{N_0(q)}_{L^2(\M_1)}^{\frac{1}{2}}\norm{N_0(q)}_{H^2(\M_1)}^{\frac{1}{2}}\leq C' \norm{N_0(q)}_{L^2(\M_1)}^{\frac{1}{2}},$$ 
the constants $C$ and $C'$ only depend on $\M$, $r$. We then apply \eqref{2.21} to get \eqref{1.20}.
\end{proof}

\section{Extension to the Neumann case}
\setcounter{equation}{0}

We explain in this section how to adapt the preceding analysis to obtain an uniqueness result for an inverse spectral problem fo the Schr\"odinger operator under Neumann boundary condition.

For $B=(A,q)\in \mathscr{B}$, define the unbounded  self-adjoint operator $\mathscr{H}_B$, acting in $L^2(\M)$ as follows
\begin{equation}\label{8.2}
\mathscr{H}_Bu=\mathcal{H}_B u=-\Delta_Au+qu,\quad u\in\mathscr{D}(\mathscr{H}_B),
\end{equation}
with  domain
\begin{equation}\label{8.3}
\mathscr{D}(\mathscr{H}_B)=\set{u\in H^1(\M),\,\, -\Delta _Au+qu\in L^2(\M),\ (\p_\nu+iA(\nu))u_{|\p M}=0}.
\end{equation}

Fix $B_\ell\in\mathscr{B}_r$, $\ell=1,2$ and  denote by $\para{\mu_{\ell,k}, \chi_{\ell,k}}$, $k\geq 1$, the eigenvalues and normalized eigenfunctions of $\mathscr{H}_{B_\ell}$.

We aim in this section to prove the following uniqueness result.

\begin{theorem}
\label{t3} 
Assume that \eqref{t3e} and the conditions 
\bel{t3b}   \sum_{k=1}^{+\infty}\norm{\chi_{1,k}-\chi_{2,k} }_{L^2(\p M)}^2<\infty,\ee
\bel{t33}\lim_{k\to+\infty}k^{-\frac{1}{n}}\abs{\mu_{1,k}-\mu_{2,k}}=0,\ee
are fulfilled.
Then $A_1^s=A_2^s$. 
\end{theorem}


Note that, according to Weyl's formula in \cite[page 114]{Ber}, we have that
$$
\lim_{k\to+\infty}k^{-\frac{1}{n}}\abs{\mu_{1,k}-\mu_{2,k}}<\infty.
$$
Therefore, condition \eqref{t33} seems to be the optimal rate of growth of the difference of eigenvalues that guaranty the uniqueness of the magnetic potential. 



\subsection{Boundary representation formulae for the Neumann problem}


For $g \in H^{1/2}(\p \M)$ and  $\rho (\mathscr{H}_B)$, consider the BVP
\bel{eq:lambdaN}
\left\{ 
\begin{array}{ll} 
(\mathcal H_B -\lambda )v  =  0\quad  &\mbox{in}\; \M ,
\\ 
(\p_\nu+iA\nu)v = g &\mbox{on}\; \p \M.
\end{array}\right.
\ee

Similarly to the Dirichlet case, for $\ell=1,2$, define the N-to-D map
\[
\mathcal N_{\ell ,\lambda}:g\in H^{{1\over 2}}(\partial \M)\mapsto{v_j(\lambda)}_{|\p \M},
\]
where $v_j(\lambda)\in H^2(\M )$ is the solution of the BVP \eqref{eq:lambdaN}. 

Define, For $\ell =1,2$, 
\begin{align}
 Q_j(\tau)&=\left\langle \mathcal N_{j,\lambda_\tau^2} (\p_{\nu}+iA_j\nu)\varphi_{1,\tau} ^*,(\p_{\nu}+iA_j\nu)\varphi_{2,\tau} ^*\right\rangle \label{Q}
 \\
 &=\int_{\p \M}(\p_{\nu}-iA_j\nu)\overline{\varphi_{2,\tau} ^*}\mathcal N_{j,\lambda_\tau^2} (\p_{\nu}+iA_j\nu)\varphi_{1,\tau} ^*\ds, \nonumber
\end{align}
with $\varphi_{j,\tau}^*$, $j=1,2$, given in \eqref{4.5}. 

\begin{proposition}\label{p1N} 
We have
\begin{align}
&Q_1(\tau)= \int_{\p \M} ({\rm i}\lambda_\tau)\p_\nu\psi \bb_1+{\rm i}(A_1\nu) \bb_1+\p_\nu \bb_1) \overline{ \bb_2(x)} \ds(x)\label{l1aN}
\\
&-2\lambda_\tau\int_{S_y(\M_1 )} \int_0^{\ell_+(y,\theta)} \widetilde{\sigma}_A(r,y,\theta)\widetilde{\beta}^*_1\overline{\widetilde{\beta}^*_2}\varrho^{1/2} \, \dd r \, \dd\omega_y(\theta)-\int_\M\bb_1\overline{\h_{B_1} (\bb_2)}\dv\nonumber
\\
&+\int_\M \left[(\mathscr{H}_{B_1}-\lambda_\tau^2)^{-1}(e^{{\rm i}\lambda_\tau\psi} \mathcal H_{A_1,q_1}\bb_1)\right]e^{-{\rm i}\lambda_\tau\psi}\left[2\lambda_\tau(A\nabla \psi)\overline{\bb_2}+ \overline{ \mathcal H_{B_1}\bb_2}\right]\dv\nonumber
\end{align}
and
\begin{align}
&Q_2(\tau)= \int_{\p \M}   ({\rm i}\lambda_\tau)\p_\nu\psi \bb_1+{\rm i}(A_1\nu) \bb_1+\p_\nu \bb_1) \overline{\bb_2(x)}\ds(x)\label{l1bN}
\\
&\hskip 7cm -\int_\M\bb_1\overline{\h_{B_2} (\bb_2)}\dv\nonumber
\\
&+\int_\M \left[(\mathscr{H}_{B_2}-\lambda_\tau^2)^{-1}e^{{\rm i}\lambda_\tau\psi}( 2\lambda_\tau(-A\nabla\psi)\bb_1+\mathcal H_{B_2}\bb_1)\right]\left(e^{-{\rm i}\lambda_\tau\psi} \overline{\mathcal H_{B_2}\bb_2}\right)\dv. \nonumber
\end{align}
\end{proposition}
\begin{proof}

Applying Green's formula, we get
$$\begin{aligned}Q_1(\tau)&=\int_\M\textrm{div}(v_1(\lambda_\tau^2)\overline{\nabla_{A_1}\varphi_{2,\tau} ^*})\dv
\\
\ &=\int_\M\left\langle \nabla_{A_1}v_1(\lambda_\tau^2),\overline{\nabla_{A_1}\varphi_{2,\tau} ^*}\right\rangle_g\dv+\int_\M v_1(\lambda_\tau^2)\overline{\Delta_{A_1}\varphi_{2,\tau} ^*}\dv
\\
\ &=-\int_\M \Delta_{A_1}v_1(\lambda_\tau^2)\overline{\varphi_{2,\tau}^*}\dv+\int_{\p \M}(\p_\nu+{\rm i}A_1\nu)v_1(\lambda_\tau^2)\overline{\varphi_{2,\tau}^*}d\sigma_g
\\
&\hskip 2cm +\int_\M v_1(\lambda_\tau^2)\overline{\Delta_{A_1}\varphi_{2,\tau} ^*}\dv
\end{aligned}$$
where $v_1(\lambda_\tau^2)$ the solution of the BVP \eqref{eq:lambdaN}, with $g=(\p_\nu+{\rm i}A_1\nu)\varphi_{1,\tau} ^*$, $\lambda=\lambda_\tau^2$, $A=A_1$, $q=q_1$.
Using the fact that $$(\p_\nu+{\rm i}A_1\nu)v_1(\lambda_\tau^2)(x)=g(x)=(\p_\nu+{\rm i}A_1\nu)\varphi_{1,\tau} ^*(x),\quad x\in\p \M,$$ we deduce that
$$\begin{aligned}Q_1(\tau)=&\int_{\p \M}   ({\rm i}\lambda_\tau)\p_\nu\psi \bb_1+{\rm i}(A_1\nu) \bb_1+\p_\nu \bb_1) \overline{\bb_2(x)}\ds(x)\\
\ &-\int_\M \Delta_{A_1}v_1(\lambda_\tau^2)\overline{\varphi_{2,\tau} ^*}\dv+\int_\M v_1(\lambda_\tau^2)\overline{\Delta_{A_1}\varphi_{2,\tau} ^*}\dv .\end{aligned}$$

This identity at hand, we proceed as in Lemma \ref{L.4.1} to get \eqref{l1aN}. Similar arguments allows us to derive \eqref{l1bN}.\end{proof}

As for the derivation of \eqref{4.39}, we obtain from \eqref{l1aN} and \eqref{l1bN} the following identity
\begin{align}
Q_2&(\tau)-Q_1(\tau) \label{repN}
\\
&= 2\lambda_\tau\int_{S_y(\M_1 )} \int_0^{\ell_+(y,\theta)} \widetilde{\sigma}_A(r,y,\theta)\widetilde{\beta}^*_1\overline{\widetilde{\beta}^*_2}\varrho^{1/2} \, \dd r \, \dd\omega_y(\theta) \nonumber
\\
& +\int_\M  (q_1-q_2)\bb_1\bb_2\dv(x) -\int_\M \bb_1\overline{(\Delta_{A_1}\bb_2-\Delta_{A_2}\bb_2)}\dv \nonumber
\\
& -\int_\M \left[(\mathscr{H}_{B_1}-\lambda_\tau^2)^{-1}(e^{{\rm i}\lambda_\tau\psi} \mathcal H_{B_1}\bb_1)\right]e^{-{\rm i}\lambda_\tau\psi}\left[2\lambda_\tau(A\nabla \psi)\overline{\bb_2}+ \overline{ \mathcal H_{B_1}\bb_2}\right]\dv. \nonumber
\\ 
&+\int_\M  \left[(\mathscr{H}_{B_2}-\lambda_\tau^2)^{-1}e^{{\rm i}\lambda_\tau\psi}( 2\lambda_\tau(-A\nabla\psi)\bb_1+\mathcal H_{B_2}\bb_1)\right]\left(e^{-{\rm i}\lambda_\tau\psi}\overline{\mathcal H_{B_2}\bb_2}\right)\dv ,\nonumber
\end{align}
from which we deduce that, for all $y\in\p \M_1$ and all $\eta\in H^2(S_y^+\M_1)$,  
\bel{l2aN}2i\int_{S_y^+(\M_1 )}\left(e^{iI_1A(y,\theta)}-1\right)\eta(y,\theta)\dd\omega_y(\theta)=\lim_{\tau\to+\infty}{Q_2(\tau)-Q_1(\tau)\over \tau}.\ee

The following lemma is needed in the proof of Theorem \ref{t3}.

\begin{lemma}\label{l44} For $\ell=1,2$, consider $\varphi_{j,\tau}^*$, $j=1,2$,  given by \eqref{4.5}. Then, we have
\bel{l4aN}\begin{aligned}\sum_{k=1}^\infty k^{\frac{2}{n}} \left|{\left\langle \varphi_{1,\tau} ^*,\chi_{\ell,k}\right\rangle \over \mu_{\ell,k}-\lambda_\tau^2 }\right|^2<C \norm{\eta}^2_{H^2(S_y^+( \M_1))}\tau^{2},\\
  \sum_{k=1}^\infty k^{\frac{2}{n}} \left|{\left\langle \overline{\varphi_{2,\tau} ^*},\chi_{\ell,k}\right\rangle \over \mu_{\ell,k}-\lambda_\tau^2 }\right|^2\leq C \tau^{2} ,\ \ell=1,2\end{aligned}\ee
  with $C>0$ independent of $\tau$.\end{lemma}
\begin{proof}  

Let $\tau =\norm{q_1}_{L^\infty(M)}+\norm{q_2}_{L^\infty(M)}+1$ and note  that $D((\mathscr{H}_{B_\ell }+\tau )^{1/2})=H^1(M)$ since it coincides with the domain of the form associated to the operator $\mathscr{H}_{B_\ell}+\tau$. Whence, for any $w\in H^1(M)$, we have
$$ \sum_{k=1}^\infty (1+|\mu_{\ell,k}|)|(w,\chi_{\ell,k})_{L^2(M)}|^2\leq C\norm{w}_{H^1(M)}^2,$$
the constant $C$ only depends on $\tau$, $A_\ell$, $q_\ell$ and $\M$. Combining this estimate with a Weyl's formula for Neumann magnetic operators, similar to that in Lemma \ref{L.A1}, we get \eqref{l4aN}.

\end{proof}

\subsection{End of the proof of Theorem \ref{t3}.}

The following lemma is useful in the sequel
\begin{lemma}\label{lem:Neum}
Let $g \in H^{1/2}(\p \M)$, $B\in \mathscr{B}$,  $\lambda \in  \rho \left(\mathscr{H}_{B}\right)$ and denote by $v(\lambda)$ the solution of the BVP \eqref{eq:lambdaN}. Then
\bel{eq:N} v(\lambda)_{|\p \M}=\sum_{k \geq 1} {\left\langle g,\chi_{k}\right\rangle \over \lambda - \mu_{k}}\,  \chi_{k},\ee
the convergence takes place in $H^{1/2}(\p \M)$.\end{lemma}


In light of this lemma, we have
\begin{align}Q_2(\tau)&-Q_1(\tau) \label{repNN}
\\
&=\sum_{k=1}^{\infty}{\left\langle (\p_\nu+iA_1\nu)\varphi_{1,\tau} ^*,\chi_{2,k} \right\rangle \left\langle \chi_{2,k},(\p_\nu+iA_1\nu)\varphi_{2,\tau} ^*\right\rangle\over \lambda_\tau^2-\mu_{2,k}} \nonumber
\\
&\qquad -{\left\langle (\p_\nu+iA_1\nu)\varphi_{1,\tau} ^*,\chi_{1,k} \right\rangle \left\langle \chi_{1,k},(\p_\nu+iA_1\nu)\varphi_{2,\tau} ^*\right\rangle\over \lambda_\tau^2-\mu_{1,k}}. \nonumber
\end{align}
Observe that, according to \eqref{t3e}, $A_1$ can be substituted by $A_2$ in the identity above. 

On the other hand, we have from \eqref{4.5}
\begin{align} 
\norm{(\p_\nu+iA_1\nu)\varphi_{j,\tau} ^*}_{L^2(\p \M)}&\leq|\lambda_\tau|\norm{\p_\nu\psi \bb_j}_{L^2(\p \M)}+\norm{(\p_\nu+iA_1\nu)\bb_j}_{L^2(\p \M)}\label{anzz2} 
\\
&\leq C\tau(1+\norm{\eta}_{H^2(S_y\M_1 )}),\nonumber
\end{align}
the constant $C$ being independent of $\tau$. Thus,
\[
|Q_2(\tau)-Q_1(\tau)|\leq \sum_{k=1}^{\infty}\mathcal E_k(\tau)+\sum_{k=1}^{\infty}\mathcal F_k(\tau)+\sum_{k=1}^{\infty}\mathcal G_k(\tau),
\]
with
\begin{align*}
\mathcal E_k(\tau)&=\norm{(\p_\nu+iA_1\nu)\varphi_{1,\tau} ^*}_{L^2(\p \M)}\norm{\chi_{2,k}-\chi_{1,k}}_{L^2(\p \M)}{|\left\langle \chi_{2,k},(\p_\nu+iA_1\nu)\varphi_{2,\tau} ^*\right\rangle|\over |\lambda_\tau^2-\mu_{2,k}|}
\\
&\leq C\norm{\chi_{2,k}-\chi_{1,k}}_{L^2(\p \M)}\tau {|\left\langle \chi_{2,k},(\p_\nu+iA_1\nu)\varphi_{2,\tau} ^*\right\rangle|\over |\lambda_\tau^2-\mu_{2,k}|},
\end{align*}

\begin{align*}
\mathcal F_k(\tau)&=\norm{(\p_\nu+iA_1\nu)\varphi_{2,\tau} ^*}_{L^2(\p \M)}\norm{\chi_{2,k}-\chi_{1,k}}_{L^2(\p \M)}{|\left\langle (\p_\nu+iA_1\nu)\varphi_{1,\tau} ^*,\chi_{1,k}\right\rangle|\over |\lambda_\tau^2-\mu_{2,k}|}
\\
&\leq C\norm{\chi_{2,k}-\chi_{1,k}}_{L^2(\p \M)}\tau {|\left\langle (\p_\nu+iA_1\nu)\varphi_{1,\tau} ^*,\chi_{2,k}\right\rangle|\over |\lambda_\tau^2-\mu_{2,k}|}+C\tau^2\norm{\chi_{1,k}-\chi_{2,k}}_{L^2(\p \M)}^2
\end{align*}
and
\begin{align*}
\mathcal G_k(\tau)&={|\left\langle (\p_\nu+iA_1\nu)\varphi_{1,\tau} ^*,\chi_{1,k}\right\rangle||\left\langle \chi_{1,k},(\p_\nu+iA_1\nu)\varphi_{2,\tau} ^*\right\rangle||\mu_{2,k}-\mu_{1,k}|\over|\lambda_\tau^2-\mu_{2,k}||\lambda_\tau^2-\mu_{1,k}|}
\\
& \leq C k^{-\frac{1}{n}}|\mu_{2,k}-\mu_{1,k}|\norm{\chi_{1,k}-\chi_{2,k}}_{L^2(\p \M)}k^{\frac{1}{n}}{|\left\langle (\p_\nu+iA_1\nu)\varphi_{1,\tau} ^*,\chi_{1,k}\right\rangle|\over|\lambda_\tau^2-\mu_{1,k}|}
\\
&+k^{-\frac{1}{n}}|\mu_{2,k}-\mu_{1,k}|k^{\frac{1}{n}}{|\left\langle (\p_\nu+iA_1\nu)\varphi_{1,\tau} ^*,\chi_{1,k}\right\rangle|\over|\lambda_\tau^2-\mu_{1,k}|}{|\left\langle \chi_{2,k},(\p_\nu+iA_1\nu)\overline{\varphi_{2,\tau}^*}\right\rangle|\over|\lambda_\tau^2-\mu_{2,k}|},
\end{align*}
the constant $C>0$ being independent on $\tau$ and $k$.

Noting that
$$\sup_{\tau >1}{{\tau^2}\over{|\lambda_\tau^2-\mu_{\ell,k}|}}<\infty,\quad \ell=1,2,\ k\geq1,$$
we deduce that we have, for all $k\geq1$, 
$$\limsup_{\tau\to+\infty} \tau^{-1}\mathcal E_k(\tau)=\limsup_{\tau\to+\infty} \tau^{-1}\mathcal F_k(\tau)=\limsup_{\tau\to+\infty} \tau^{-1}\mathcal G_k(\tau)=0.$$
Then, for any arbitrary integer $n_1\geq1$, we get
\begin{align*}
&\limsup_{\tau\to+\infty} \tau^{-1}\sum_{k=1}^{\infty}\mathcal E_k(\tau)\leq\limsup_{\tau\to+\infty} \tau^{-1}\sum_{k=n_1}^{\infty}\mathcal E_k(\tau),
\\
&\limsup_{\tau\to+\infty} \tau^{-1}\sum_{k=1}^{\infty}\mathcal F_k(\tau)\leq\limsup_{\tau\to+\infty} \tau^{-1}\sum_{k=n_1}^{\infty}\mathcal F_k(\tau),
\\
&\limsup_{\tau\to+\infty} \tau^{-1} \sum_{k=1}^{\infty}\mathcal G_k(\tau)\leq\limsup_{\tau\to+\infty} \tau^{-1}\sum_{k=n_1}^{\infty}\mathcal G_k(\tau).
\end{align*}
We combine these inequalities,  estimates  \eqref{l4aN} and Weyl's formula in order to get, by repeating the arguments used to prove Lemma \ref{L.5.2}, that
$$\limsup_{\tau\to+\infty}\left|{Q_2(\tau)-Q_1(\tau)\over\tau}\right|\leq C(1+\norm{\eta}_{H^2(S_y^+(\M_1))})^2(\limsup_{k\to+\infty}k^{-\frac{1}{n}}|\mu_{2,k}-\mu_{1,k}|).$$
Then, from  \eqref{t33} and \eqref{l2aN}  we deduce that $\I_1A\in 2\pi\mathbb Z$. We proceed similarly to the proof of Theorem \ref{t1} to get that $A_1^s=A_2^s$.  This completes the proof of Theorem \ref{t3}.

\appendix 
\section{Weyl's formula}\label{appendix}
\setcounter{equation}{0}
We establish some uniform estimates related  to the Weyl's formula of magnetic Schr\"odinger operators. Our estimates, which are also valid for the Neuman realization of magnetic Schr\"odinger operators, can be stated as follows.
\begin{lemma}\label{L.A1} 
Let $B=(A,q)\in \mathscr{B}$. Then there exists a constant $C>1$, only depending on $\M$ and $r\ge \norm{A}_{L^\infty(\M,T^\ast \M)}^2+\norm{q}_{L^\infty(\M)}$ so that
\begin{equation}\label{A1} 
C^{-1}k^{2/n}\leq 1+|\lambda_B^k|\leq C k^{2/n},\quad k\geq1
\end{equation}
\end{lemma}
\begin{proof} 
Let $(\lambda_k )$ be the sequence of eigenvalues, counted according to their multiplicities, of the Laplace-Belrami operator under Dirichlet boundary condition. By Weyl's asymptotic formula \cite[page 114]{Ber}
\begin{equation}\label{A2}
\lambda_k=\mathcal{O}\left( k^{\frac{2}{n}}\right),\quad  k\ge 1.
\end{equation}
The sesquilinear form associated to $H_{B}$ is given by
$$
\mathbf{a}(u,v)=\int_\M \seq{\nabla_A u,\overline{\nabla_Av}}\dv+\int_\M qu\overline{v}\dv ,\quad  u,v\in H^1_0(\M).
$$
Then it is not hard to check that
$$
\begin{aligned}
\mathbf{a}(u,u)&\leq \norm{d u}_{L^2(\M)}^2+2\sqrt{r}\norm{u}_{L^2(\M)}\norm{d u}_{L^2(\M)}+r\norm{u}_{L^2(\M)}^2\\
\ &\leq\frac{3}{2}\norm{d u}_{L^2(\M)}^2+r\norm{u}_{L^2(\M)}^2
\end{aligned}
$$
and
$$
\begin{aligned}
\mathbf{a}(u,u)&\geq \norm{d u}_{L^2(\M)}^2-2\sqrt{r}\norm{u}_{L^2(\M)}\norm{d u}_{L^2(\M)}-r\norm{u}_{L^2(\M)}^2
\\
\ &\geq\frac{1}{2}\norm{d u}_{L^2(M)}^2-3r\norm{u}_{L^2(\M)}^2.
\end{aligned}
$$
We get the expected two-sided inequalities \eqref{A1} by using \eqref{A2} and the minmax principle. 
\end{proof}


\end{document}